\documentclass[11pt]{article}
\usepackage[margin=1in, a4paper]{geometry}
 \pdfoutput=1
\usepackage{amsthm, amsfonts,amsmath,amssymb,bbm, bbold}
\usepackage{mleftright}
\usepackage{verbatim}
\usepackage{textgreek}
\usepackage[]{authblk}
\usepackage{color}
\usepackage{hyperref}
\usepackage{ytableau}
\usepackage{natbib}

\usepackage{etoolbox}
\makeatletter
\let\ams@starttoc\@starttoc
\makeatother
\usepackage[parfill]{parskip}
\makeatletter
\let\@starttoc\ams@starttoc
\patchcmd{\@starttoc}{\makeatletter}{\makeatletter\parskip\z@}{}{}
\makeatother
\usepackage{mdframed}

\usepackage{enumerate}
\usepackage{pdfsync}
\usepackage{color}
\usepackage{float}

\usepackage[justification=centering]{caption}

\numberwithin{equation}{section}

\newtheorem{theorem}{Theorem}[section]
\newtheorem{proposition}[theorem]{Proposition}
\newtheorem{lemma}[theorem]{Lemma}
\newtheorem{corollary}[theorem]{Corollary}

\theoremstyle{definition}

\newtheorem{remark}[theorem]{Remark}
\newcommand{\pmodsmall}[1]{(\mathrm{mod}\,#1)}

\let\cal\mathcal

\newcommand{\lb}{\mleft(}
\newcommand{\rb}{\mright)}
\newcommand{\lbb}{\mleft [}
\newcommand{\rbb}{\mright ]}
\newcommand{\labs}{\mleft |}
\newcommand{\rabs}{\mright |}
\newcommand{\lbrb}[1]{\lb #1 \rb}

\newcommand{\labsrabs}[1]{\labs#1\rabs}

\newcommand{\langlerangle}[1]{\langle#1\rangle}
\newcommand{\lbcurly}{\mleft\{}
\newcommand{\rbcurly}{\mright\}}
\newcommand{\lbcurlyrbcurly}[1]{\lbcurly#1\rbcurly}

\usepackage{mathrsfs} 
\newcommand{\MoreIt}[1]{%
  \tikz[baseline=(X.base)] \node[inner sep=0, xslant=0.3] (X) {$#1$};%
}
\newcommand{\Lat}{\ensuremath{\mathscr{L}\!\!\MoreIt{at}}}

\newcommand{\cA}{\mathcal{A}}
\newcommand{\cB}{\mathcal{B}}
\newcommand{\cH}{\mathcal{H}}
\newcommand{\one}{\boldsymbol{1}}

\newcommand{\simi}{\stackrel{\infty}{\sim}}

\newcommand{\dij}[2]{\delta_{#1,#2}}
\newcommand{\dijm}[3]{\delta_{#1,#2}^{(#3)}}

\newcommand{\Ebb}[1]{\Eb\lbb #1\rbb}

\newcommand{\Eb}{\mathbb{E}}

\newcommand{\Rb}{\mathbb{R}}
\newcommand{\R}{\mathbb{R}}

\newcommand{\Qb}{\mathbb{Q}}

\newcommand{\Zb}{\mathbb{Z}}
\renewcommand{\P}{\mathbb{P}}

\newcommand{\Bc}{\mathcal{B}}

\newcommand{\Hc}{\mathcal{H}}

\newcommand{\Xc}{\mathcal{X}}

\newcommand{\supp}{{\rm{Supp}}}

\newcommand{\ind}[1]{\mathbbm{1}_{\lbcurlyrbcurly{#1}}}

\usepackage{multicol}
\usepackage{blindtext}
\usepackage{amsmath}

\usepackage{xfrac} 
\usepackage{amsmath,amssymb,booktabs} 
\usepackage{nicematrix}
\usepackage{tikz}
\usepackage{nicematrix}
\usepackage{booktabs}
\usepackage{eqparbox}

\newcommand{\eqinfo}[1]{\stackrel{
\eqmakebox[a][c]{\scriptsize #1}}{=}}

\DeclareMathOperator{\Span}{span}
\DeclareMathOperator{\lcm}{lcm}
\DeclareMathOperator{\rank}{rank}

\newcommand{\shadehalf}[3]{%
    \raisebox{0pt}[0pt][0pt]{%
        \begin{tikzpicture}[baseline=1, scale=0.8] 
            \def\leng{0.65}    
            \def\xone{0}    
            \def\yone{-0.16}   
            \pgfmathsetmacro{\xtwo}{\xone + \leng}    
            \pgfmathsetmacro{\ytwo}{\yone + \leng}    
            
            \fill[#1] (\xone,\yone) rectangle (\xtwo,\ytwo); 
            \fill[#2] (\xone,\yone) -- (\xtwo,\yone) -- (\xtwo,\ytwo) -- cycle; 
            
            \node at ({(\xtwo+\xone)/2},{(\yone+\ytwo)*0.6}) {#3};
        \end{tikzpicture}%
    }
}
 \newcommand{\indep}{\perp\!\!\!\!\perp} 
\usepackage{authblk}

\DeclareMathOperator{\Cov}{Cov}
\DeclareMathOperator{\Var}{Var}

\makeatletter
\newcommand{\myitem}[1]{%
	\item[#1]\protected@edef\@currentlabel{#1}%
}
\makeatother

\usepackage{tikz}
\usetikzlibrary{decorations.markings}
\usetikzlibrary{patterns}
\usetikzlibrary{calc} %

\DeclareMathSymbol{\mrq}{\mathord}{operators}{`'}
\usepackage{wrapfig}

\allowbreak

\renewcommand{\Im}{\mathtt{Im}}

\let\cal\mathcal

\begin{document}

        \author[1,2]{Stefan Gerdjikov}
    \author[,1]{Martin Minchev\thanks{
			Email: mjminchev@fmi.uni-sofia.bg.}}
	\author[1,3]{Mladen Savov}
	
	\affil[1]{Faculty of Mathematics and Informatics, Sofia University "St. Kliment Ohridski", 5,
		James Bourchier blvd., 1164 Sofia, Bulgaria}
        
        	\affil[2]{Institute for Information and Communication Technologies, Bulgarian Academy of Sciences, Akad.  Georgi Bonchev str., Block 2, Sofia  1113, Bulgaria}
	
	\affil[3]{Institute of Mathematics and Informatics,  Bulgarian Academy of Sciences, Akad.  Georgi Bonchev str., 
		Block 8, Sofia 1113, Bulgaria}
	\title{On the probability of $n$
    equidistant points in high-dimensional lattices}
	\maketitle

	\begin{abstract}
Consider $n$ $d$-dimensional vectors with
iid entries from a lattice distribution $X$. We show that the probability that all distances between them are equal is
asymptotically
\[
C_n\cdot\frac{1}{d^{(m-1)/2}}
\quad
\text{for}
\quad
d \to \infty
\quad
\text{and}
\quad
m = \binom{n}{2},
\]
with an explicit constant in terms of the first 4 moments of $X$. Moreover, we generalise this result to encompass all 
 finitely supported $X$, as well as 
 under different distances.
Our method relies on the
relatively rarely used multidimensional local limit theorem and an analysis of the lattice on $\Zb^{\binom{n}{2}}$
spanned by the image of the \emph{overlapping} map
\[
H : (v_1, \dots, v_n) \in \{0,1\}^n \mapsto 
\lbrb{\ind{v_i \neq v_j}}_{1 \leq i < j\leq n}
\in \{0,1\}^{\binom{n}{2}}.
\]

	\end{abstract}

	
	
	\section{Introduction}\label{sec:intro}
    
    Let $\boldsymbol{X}_1,  \dots, \boldsymbol{X}_n$ be $d$-dimensional vectors
 whose entries are iid
 samples from a random variable $X$
 with finite lattice support. Let 
$p_d$ be the probability that these vectors are equidistant with respect to the Euclidean distance.
Then,  if we denote by  \[
m = \binom{n}{2},
\quad
\text{and}
\quad
C_1 := \Var\lbrb{\lbrb{X- \Ebb{X}}^2},
\]
in Theorem \ref{thm: general lattice}, we show  that as $d \to \infty$,
\[
 p_d \simi
\frac{1}{d^{(m-1)/2}}\cdot
\Big(
m(2\pi)^{m-1}(\Var X)^{2(m-n)}(4(\Var X)^2+(n-2)C_1)^{n-1}
\Big)^{-1/2}
\]
The method is not entirely specific to the choice of distance, as discussed in Remark \ref{rem: metric}. Furthermore,
we generalise the result to all finitely supported $X$ in Theorem \ref{thm: general}.

 Let us briefly present the case of $n=3$ vectors and $X\sim Ber(1/2)$ to illustrate some of the techniques
 and outline potential challenges for the general case. Denoting by $d_H$ the Hamming (which coincides with the Euclidean) distance between binary vectors, we are interested in the probability of
 the event
 \[
 d_H\lbrb{\boldsymbol{X}_{\! 1}, \boldsymbol{X}_{\! 2}} = 
  d_H\lbrb{\boldsymbol{X}_{\! 1}, \boldsymbol{X}_{\! 3}} =
   d_H\lbrb{\boldsymbol{X}_{\! 2}, \boldsymbol{X}_{\! 3}}.
 \]
 Writing this as a system of two equations
 \begin{equation}\label{eq: systems}
  \begin{cases}
 d_H\lbrb{\boldsymbol{X}_{\! 1}, \boldsymbol{X}_{\! 3}}=
  d_H\lbrb{\boldsymbol{X}_{\! 1}, \boldsymbol{X}_{\! 2}}   \\
  d_H\lbrb{\boldsymbol{X}_{\! 2}, \boldsymbol{X}_{\! 3}} =   d_H\lbrb{\boldsymbol{X}_{\! 1}, \boldsymbol{X}_{\! 2}}
\end{cases}
\text{or equivalently}
\quad
\begin{cases}
    \sum_{\ell = 1}^d
    \lbrb{ X_1^{(\ell)} - X_3^{\!(\ell)}}^2 - \lbrb{ X_1^{(\ell)} - X_2^{\!(\ell)}}^2 =0 \\
 \sum_{\ell = 1}^d
    \lbrb{ X_2^{(\ell)} - X_3^{\!(\ell)}}^2 - \lbrb{ X_1^{(\ell)} - X_2^{\!(\ell)}}^2 =0
\end{cases}
\end{equation}
we see that the last system can be written in vector form as
\[
\sum_{
\ell=1}^d \boldsymbol{V}_{\!\ell} = \boldsymbol{0}
\]
with iid vectors $\boldsymbol{V}_{\!i}\sim \boldsymbol{V}$. The last representation suggests that the asymptotics of the respective probability
can be obtained by applying
a multidimensional local limit theorem.
Such type of results were first obtained in the 1950s, e.g. in
\cite{Meizler-1949}, and large deviations
results further available in the works of 
\cite{Richter-1958,Steinebach-1978, Ellis-1984, Chaganty-1986}. In our context, we must take into account the lattice nature of $X$, and using a result of 
\cite{Krafft-1967}, see \eqref{thm: Krafft}, we have that
\begin{equation}
    \label{eq: toy}
\P \lbrb{ \sum_{\ell=1}^d \boldsymbol{V}_{\!\ell} = \mathbf{0} }
\simi 
\frac{\labsrabs{ \Lat \boldsymbol{V}}}{2\pi d\sqrt{\labsrabs{\Var(\boldsymbol{V})}}},
\end{equation}
with $|\Var(\boldsymbol{V})|$
being the determinant of the covariance matrix of the vector $\boldsymbol{V}$, and $|\Lat \boldsymbol{V}|$ the minimal non-zero volume of a parallelepiped from the lattice spanned by $\supp \boldsymbol{V}$ (also known as the \emph{fundamental volume}), see Section \ref{sec: Lattice} for a more formal definition.
In our toy example, observe that if   $(X_1,X_2,X_3) = (1,0,1)$, then $\boldsymbol{V}=
(-1, 0)$, and if $(X_1,X_2,X_3) = (1,0,0)$, $\boldsymbol{V} = (0,-1)$, so 
$\Lat \boldsymbol{V} \equiv \Zb^2$, and $|\Lat \boldsymbol{V} |= 1$. 

However, if we consider more than three vectors, or $X$ taking values from a  different set than $\{0,1\}$, the structure of  $\Lat \boldsymbol{V}$ becomes non-trivial but surprisingly universal. To address this problem, in Section \ref{sec: Lattice} we study the lattice on
$\Zb^n$ spanned by the image of the \emph{overlapping} map
\begin{equation}\label{def: overlap}
H : (v_1, \dots, v_n) \in \{0,1\}^n \mapsto 
\lbrb{\ind{v_i \neq v_j}}_{1 \leq i < j\leq n}
\in \{0,1\}^{\binom{n}{2}}.
\end{equation}
We provide a basis of $\Lat H :=\Lat \Im H$ in
Theorem \ref{thm: lattice}, and use the arguments made to construct also a basis of $\Lat\boldsymbol{V}$
in Lemma \ref{lem: H^ast}.
It turns out that the basis
built for the binary case is in fact the
same modulo scaling for all lattice $X$, which is described in detail in the proof of Theorem \ref{thm: general lattice}. 

Furthermore, consider what happens, for example, if $\supp X = \{0, 1,\sqrt{2}\}$. Then the linear span over $\Zb$ of $\supp X$ is dense in $\Rb$, and
\emph{a fortiori} $|\Lat \boldsymbol{V}| = 0$. In this case, we embed the problem in a higher dimensional space by picking a suitable Hamel basis of this span over $\Qb$. The analysis of $\Lat \boldsymbol{V}$ in this case is done in
Lemma \ref{lem: lattice_r}.

The second problem in determining the exact 
asymptotics lies in calculating $|\Var(\boldsymbol{V})|$. In the toy example, we can do this by hand, but in the general case, we need to
understand in depth the nature of the apparent overlap in the structure of systems of equations that we construct in the style of \eqref{eq: systems}. 
This is in fact related to the overlapping map $H$: denoting by $e_i$ the standard basis of $\R^n$, let us consider the $m:=\binom{n}{2}$ vectors,
for $i<j$,
\begin{equation}
    \label{def: cH}
\cH_{i,j} := H(e_i + e_j).
\end{equation}
One such vector $\cH_{i,j}$ has 1s
at positions $(a,b)$ such that
$\{a,b\}$ and $\{i,j\}$ have exactly 1 element in common, and it can be expected to have a similar nature to elements of the type
\[
\Cov\lbrb{(X_i-X_j)^2 - (X_1-X_2)^2,
(X_i-X_k)^2 - (X_1-X_2)^2}.
\]
Let us order $\cH_{i,j}$ in an $m\times m$ matrix $\cH$, which is indexed by pairs $(a,b)$ ordered in the usual lexicographic order.
This matrix is a known object in graph theory: consider a graph $G(V,E)$. Then, its \emph{line graph}
$L(G)$ is the graph with vertex set $E$, and such that two of its
vertices are connected if
the respective edges in $G$ have 
a point in common. With this notion, 
the matrix $\cH$ 
is the adjacency matrix of the {line graph} of the
complete graph. 
In order to calculate
explicitly $|\Var(\boldsymbol{V})|$,
in Theorems \ref{thm: general lattice} and \ref{thm: general}, we use the understanding of the eigenstructure of $\cH$ as outlined in Proposition \ref{prop: line Kn}. Related results may be found for example in \cite[Chapter 6]{Godsil-Meagher-2016}, but in the context of
general Johnson graphs, and thus being  relatively implicit for our needs.

To link the problem to previous research, we note that in the binary case, or more generally 
if $\supp X = \{0,1,2, \dots, q-1\}$, the problem is equivalent to finding the asymptotic density of equidistant codes (or for $q=2$, equivalently equidistant families of sets) with length $d$ and cardinality $n$, for which there is no explicit formula. In the binary case, \cite{Ionin-1995} show that equidistant codes are in bijection with certain types of Hadamard matrices, and characterise the maximal possible distance between two words. Also, from a purely combinatorial point, \cite{Hudelson-1996} studies maximum-volume $j$-dimensional simplices (and thereby certain large equidistant sets) in 
$\{0,1\}^d$, relates them to Hadamard matrices, and includes partial enumeration and bounding results in special cases.

    \section{The lattice spanned by the overlapping map \texorpdfstring{$H$}{H}
    }\label{sec: Lattice}

In order to apply the multidimensional local limit theorem, 
we need to understand the lattice spanned by the respective vectors. We briefly present key notions and refer to
\cite[Chapter 3.1]{Tao-Vu-2006} or the notes \cite[Chapter 2]{Paffenholz-2022}
for a more
detailed overview.

We will call a  \emph{lattice} in $\R^d$ an additive subgroup $\Lambda \subseteq \R^d$, whose points are isolated. This is also known as
a \emph{discrete} lattice,
and we call the dimension of its linear span the \emph{rank} of the lattice. It is true that if $A = \{a_1, \dots, a_n\} \subseteq \Qb^d$ is a set of  vectors,  then
\[
\Lat A := \lbcurlyrbcurly{\sum_{i=1}^m \lambda_i a_i : 
\lambda_i \in \Zb, 1 \leq i \leq m}
\]
is a lattice, which we call \emph{generated} or also
\emph{induced} by $A$.  Because the problem of equidistant vectors is invariant with respect to translation and dilation of the constituting random variables, we will work hereafter with lattices of integer points.

Similarly to the usual notion, we call a set $B$ a basis of $\Lat A$ if it generates $\Lat A$
and is composed of linearly independent vectors.
All bases $B$ of a lattice are equivalent modulo a unimodular transformation,
and in particular $\det B$ does not depend on the particular choice of $B$ but only on the lattice, see \cite[Corollary 2.17]{Paffenholz-2022}.
The quantity $\det B$ is known as the
\emph{fundamental volume} of $\Lat A$,
and we will denote it by $|\Lat A|$.

Recall the overlapping map 
\[
H : (v_1, \dots, v_n) \in \{0,1\}^n \mapsto 
\lbrb{\ind{v_i \neq v_j}}_{1 \leq i < j\leq n}
\in \{0,1\}^{\binom{n}{2}}.
\]
With a slight abuse of notation, using the standard notion of a characteristic vector, we can also consider that $H$ takes as input subsets $I$ of $\{1,\dots,n\}$.

 To describe the values of $H$, we will work with the standard basis
of $\Rb^{\binom{n}{2}}$, $e_1, \dots, e_{\binom{n}{2}}$, however, as the natural enumeration in our context is by pairs, we define $e_{i,j} := e_{i + \sum_{k=1}^{j-i-1}n-k}$, which is the usual diagonal order ($\searrow$). Therefore, it is convenient to order the values of $H$
in a table and not as a usual vector in a line.  Let us  define \emph{integer intervals} as
sets 
\[[i:j] :=\{k \in \Zb: i \leq k \leq j\}.
\]
The proposed triangular representation is illustrated in Figure \ref{fig: H123} on the intervals $[1:1], 
\dots, [1:n-1]$.
To simplify the exposition, \textbf{we will henceforth assume that $i<j$}.

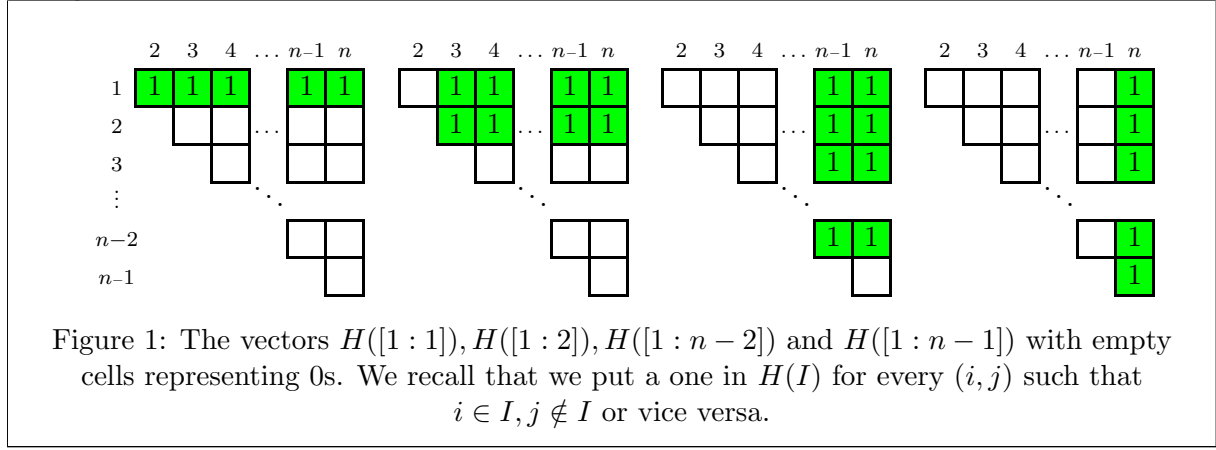
\begin{mdframed}[linecolor=black, linewidth=0.5pt, roundcorner=5pt]
    \begin{figure}[H]
    \centering
    \scalebox{1}{%
    \begin{tabular}{r@{}l}
    \ytableausetup {mathmode, boxframe=0.1em, boxsize=1.2em}  

    \begin{ytableau}
    \none & \none[\scriptstyle 2] & \none[\scriptstyle 3] & \none[\scriptstyle 4] & \none[\scalebox{0.75}{\dots}] & \none[\scriptstyle n\scalebox{0.5}{$-$}1] & \none[\scriptstyle n] \\
    \none[\scriptstyle 1] & *(green)1 & *(green)1 & *(green)1 & \none & *(green)1 & *(green)1 \\
    \none[\scriptstyle 2] & \none &  &  & \none[\scalebox{0.75}{\dots}] & &  \\
    \none[\scriptstyle 3] & \none & \none &  & \none &  &  \\
    \none[\scalebox{0.75}{\vdots}] & \none & \none & \none & \none[{\ddots}] & \none \\
    \none[\scriptstyle n-2] & \none & \none & \none & \none &  &  \\
    \none[\scriptstyle n\scalebox{0.5}{$-$}1]  & \none & \none & \none & \none & \none &  \\        
    \end{ytableau}

    \hspace{0.5em}

    \begin{ytableau} 
    \none[\scriptstyle 2] & \none[\scriptstyle 3] & \none[\scriptstyle 4] & \none[\scalebox{0.75}{\dots}] & \none[\scriptstyle n\scalebox{0.5}{$-$}1]  & \none[\scriptstyle n] \\
    & *(green)1 & *(green)1 & \none & *(green)1 & *(green)1 \\
    \none & *(green)1 & *(green)1 & \none[\scalebox{0.75}{\dots}] & *(green)1 & *(green)1 \\
    \none & \none &  & \none &  &  \\
    \none & \none & \none & \none[{\ddots}] & \none \\
    \none& \none & \none & \none &  &  \\
    \none& \none & \none & \none & \none &  \\
    \end{ytableau}

    \hspace{0.5em}

    \begin{ytableau}
    \none[\scriptstyle 2] & \none[\scriptstyle 3] & \none[\scriptstyle 4] & \none[\scalebox{0.75}{\dots}] & \none[\scriptstyle n\scalebox{0.5}{$-$}1]  & \none[\scriptstyle n] \\
    & &  & \none & *(green)1 & *(green)1 \\
    \none & & & \none[\scalebox{0.75}{\dots}]& *(green)1 & *(green)1 \\
    \none & \none &  & \none & *(green)1 & *(green)1 \\
    \none & \none & \none & \none[{\ddots}] & \none \\
     \none& \none & \none & \none & *(green)1 &*(green)1  \\
    \none& \none & \none & \none & \none &  \\        
    \end{ytableau}

    \hspace{0.5em}

    \begin{ytableau}
    \none[\scriptstyle 2] & \none[\scriptstyle 3] & \none[\scriptstyle 4] & \none[\scalebox{0.75}{\dots}] & \none[\scriptstyle n\scalebox{0.5}{$-$}1] & \none[\scriptstyle n] \\
    & & & \none & & *(green)1 \\
    \none & & & \none[\scalebox{0.75}{\dots}] & & *(green)1 \\
    \none & \none & & \none & & *(green)1 \\
    \none & \none & \none & \none[{\ddots}] & \none \\
    \none& \none & \none & \none & & *(green)1 \\
    \none& \none & \none & \none & \none & *(green)1 \\        
    \end{ytableau}
    \end{tabular}%
    } 
    \caption{The vectors $H([1:1]), H([1:2]), H([1:n-2])$ and $H([1:n-1])$ with empty cells representing $0$s. We recall that we put a one in $H(I)$ for every $(i,j)$ such that $i \in I, j\notin I$ or vice versa.}
    \label{fig: H123}
    \end{figure}
        
\end{mdframed}

\begin{theorem}\label{thm: lattice}
    Let $H$ be the map from \eqref{def: overlap} or alternatively defined for a set
     $I\subset [1:n]$ by
    \begin{equation*}
    H(I)=\sum_{i\in I, j\notin I} e_{i,j}.
    \end{equation*}
    Then the vectors 
    \begin{itemize}
        \item $H([1:1]), H([1:2]), \dots, H([1:n-1])$
        \item alongside all of the form $2e_{i,j}$ with  $j<n$
    \end{itemize}
    form a basis of the lattice induced by $\Im H$, that is $\Lat H := \Lat \Im H$.
        
         In addition, each vector of type $2e_{i,j}$  from the basis above  can be substituted with the vector $2\cdot\one = (2, \dots, 2)$.
          
\end{theorem}

\begin{proof}
    We first show that  the chosen vectors span $\Im H$, and then that
    they are
    linearly independent as well as indeed elements of 
    $\Im H$.

    The key idea in the proof of this theorem is to split to a set $I\subset [n]$ into intervals of the form $[i:j] := \{i, i+1, 
    \dots, j\} = [1:j]\setminus[1:i-1]$. Indeed,  recall that $H(I)$ includes $e_{i,j}$ if  exactly one of $i$ and $j$ is in $I$. Therefore if we represent as a union of disjoint intervals
    \begin{align*}
        I &= I_1  \cup \dots \cup I_{n_1}
       \quad 
      \text{and} \quad  I^c := [1:n]\setminus I= \widetilde{I}_1 \cup \dots \cup\widetilde{I}_{n_2} ,
    \end{align*} 
    then $H(I)$ contains elements $e_{i,j}$ with $\{i,j\}$ from $I_k\times\widetilde{I}_{k'}
   $ for some $k,k'$.
   Thus, as we work with intervals, it is 
   natural to try to express
    $H(I)$ in terms of $H([1:1]), \dots, H([1:n-1])$. 

   The main formula which helps for this is, 
   denoting by $\star$ the
   ''unordered Cartesian product'' 
   \[
   A\star  B:= \lbcurlyrbcurly{\{a,b\}: a \in A \text{ and } b \in B},
   \]
   that for disjoint sets $A$ and $B$,
   \begin{equation}\label{eq: vkl izkl}
     H(A) + H(B) - H(A \cup B) = 2 \sum_{ 
     \{i, j\} \in A\star B } e_{i,j}.
   \end{equation}
   Indeed, 
    formula \eqref{eq: vkl izkl} holds since 
    \begin{align*}
        H(A) + H(B) - H(A\cup B) &= 
        \sum_{\{i,j\}\in A \star A^c } e_{i, j}
        +\sum_{\{i,j\}\in B \star B^c } e_{i, j}
        - 
        \sum_{\{i,j\} \in (A\cup B) \star (A^c\cap B^c) } e_{i, j}
    \\
    &=\sum_{\{i,j\}\in A \star (A^c\setminus B^c) } e_{i, j}
        +\sum_{\{i,j\}\in B \star (B^c\setminus A^c) } e_{i, j},
    \end{align*}
and because we have picked disjoint $A$ and $B$, it holds that
$A^c \setminus B^c = A^c\cap B = B$, and similarly
$B^c \setminus A^c = A$.
Formula \eqref{eq: vkl izkl} is  neatly illustrated in
Figure \ref{fig: ij} and Figure \ref{fig: intervals} for the simple cases
    \begin{equation}
        \label{eq: Hi+Hj - Hij}
    H\lbrb{\{i\}} + H(\{j\}) - H(\{i,j\}) = 2 e_{i, j},
    \end{equation}
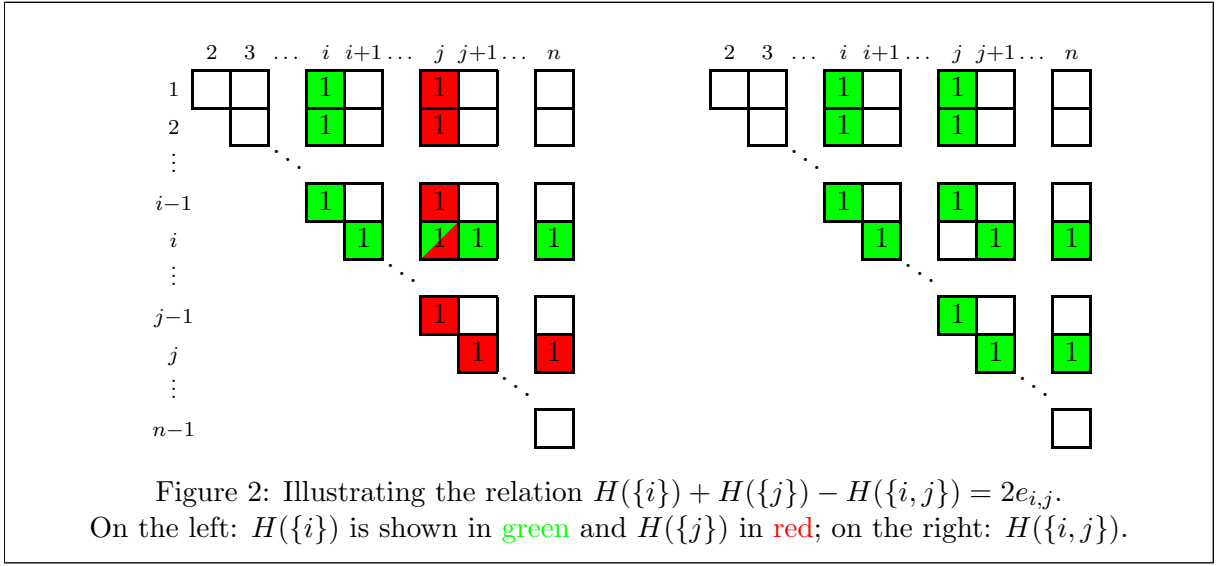
\begin{mdframed}[linecolor=black, linewidth=0.5pt, roundcorner=5pt]
\begin{figure}[H]
    \centering 
    \scalebox{1}{%
        \begin{tabular}{r@{}l}
        \ytableausetup {mathmode, boxframe=0.1em, boxsize=1.2em}  
        
        \begin{ytableau}
            \none & \none[\scriptstyle 2] & \none[\scriptstyle 3] & \none[\scalebox{0.75}{\dots}] & \none[\scriptstyle i]  & \none[\scriptstyle i+1] &\none[\scalebox{0.75}{\dots}]
            &\none[\scriptstyle j]  & \none[\scriptstyle j+1]&
            \none[\scalebox{0.75}{\dots}] & \none[\scriptstyle n] \\
            \none[\scriptstyle 1] & &  & \none & *(green)1 & &\none[] &*(red) 1 & &\none[]& 
            \\
            \none[\scriptstyle 2] & \none[]&  & \none & *(green)1 & &\none[] &*(red) 1 & &\none[]& 
            \\
            \none[\scalebox{0.75}{\vdots}] & \none & \none & \none[{\ddots}] & \none & \none \\
            \none[\scriptstyle i-1] & \none[]&   \none[]& \none & *(green)1 & &\none[] & *(red) 1& &\none[]& 
            \\
            \none[\scriptstyle i] & \none[]&   \none[]& \none & \none[]&*(green)1 &\none[] &   \shadehalf{green}{red}{1}&*(green)1 &\none[]& *(green)1
            \\
            \none[\scalebox{0.75}{\vdots}] & \none & \none & \none & \none& \none&\none[{\ddots}] &\none \\
            \none[\scriptstyle j-1] & \none & \none & \none & \none & \none &  \none&*(red) 1&&\none&\\
            \none[\scriptstyle j] & \none & \none & \none & \none & \none &  \none&\none&*(red) 1&\none&*(red) 1\\     
            \none[\scalebox{0.75}{\vdots}] & \none & \none & \none & \none& \none&\none&\none&\none&\none[{\ddots}] &\none \\
            \none[\scriptstyle n-1] & \none & \none & \none & \none & \none &  \none&\none&\none&\none&
        \end{ytableau}

        \hspace{4em}
        \begin{ytableau}
            \none[\scriptstyle 2] & \none[\scriptstyle 3] & \none[\scalebox{0.75}{\dots}] & \none[\scriptstyle i]  & \none[\scriptstyle i+1] &\none[\scalebox{0.75}{\dots}]
            &\none[\scriptstyle j]  & \none[\scriptstyle j+1]&
            \none[\scalebox{0.75}{\dots}] & \none[\scriptstyle n] \\
            &  & \none & *(green)1 & &\none[] & *(green)1  & &\none[]& 
            \\
            \none[]&  & \none & *(green)1 & &\none[] &  *(green)1 & &\none[]& 
            \\
            \none & \none & \none[{\ddots}] & \none & \none \\
            \none[]&   \none[]& \none & *(green)1 & &\none[] & *(green)1  & &\none[]& 
            \\
            \none[]&   \none[]& \none & \none[]&*(green)1 &\none[] & &*(green)1 &\none[]& *(green)1
            \\
            \none & \none & \none & \none& \none & \none[{\ddots}] \\
            \none & \none & \none & \none & \none &  \none& *(green)1 &  &\none& \\
            \none & \none & \none & \none & \none &  \none&\none&*(green)1 &\none&*(green)1 \\     
            \none & \none & \none& \none e & \none& \none & \none& \none&\none[{\ddots}] \\
            \none & \none & \none & \none & \none &  \none&\none&\none&\none&
        \end{ytableau}

        \end{tabular}%
    } 

    \caption{ Illustrating the relation \( H(\{i\}) + H(\{j\}) - H(\{i,j\}) = 2 e_{i, j} \). \\
    On the left: \( H(\{i\}) \) is shown in \textcolor{green}{green} and \( H(\{j\}) \) in \textcolor{red}{red}; 
    on the right: \( H(\{i,j\}) \).}
    \label{fig: ij}
\end{figure}
\end{mdframed}
    and
    \begin{equation}
        \label{eq: interval}
    H([i+1:j] = H([1:j]) - H([1:i]) + 2 \sum_{k = 1}^i \sum_{\ell=i+1}^j e_ {k,\ell}.
    \end{equation} 
    In particular, equation \eqref{eq: interval} shows how to express the image of an interval in terms of $H([1:1]), \dots, H[1:n-1], H[1:n] = \mathbf{0}$ and these of type $2e_{i,j}$. We note that
    we have picked in the theorem statement
    only $2e_{i,j}$ with $j<n$, but as it is
    seen from Figure \ref{fig: H123},
\[2e_{i,n} = 2H[1:i] - 2H[1:i-1] - \sum_{k=i+1}^{n-1}2e_{i,k} + 
\sum_{k=1}^{i-1} 2e_{k,i},
\]
so the chosen vectors describe all $2e_{i,j}$ and, in view of \eqref{eq: interval}, all  intervals.   
    \begin{mdframed}[linecolor=black, linewidth=0.5pt, roundcorner=5pt]
    
    \begin{figure}[H]
    \centering
    \scalebox{1}{%
    \begin{tabular}{r@{}l}
    \ytableausetup {mathmode, boxframe=0.1em, boxsize=1.2em}  

    \begin{ytableau}
    \none & \none[\scriptstyle 2] & \none[\scriptstyle 3] & \none[\scalebox{0.75}{\dots}] & \none[\scriptstyle i]  & \none[\scriptstyle i+1] 
    & \none[\scriptstyle i+2]&\none[\scalebox{0.75}{\dots}]
    &\none[\scriptstyle j]  & \none[\scriptstyle j+1]
     & \none[\scriptstyle j+2]&
    \none[\scalebox{0.75}{\dots}] & \none[\scriptstyle n] \\
    \none[\scriptstyle 1] & &  & \none &  &*(green)1 &*(green)1&\none[] &*(green)1 &   \shadehalf{green}{red}{1}&\shadehalf{green}{red}{1}&\none[]&\shadehalf{green}{red}{1} 
    \\
    \none[\scriptstyle 2] & \none[]&  & \none &  &*(green)1 &*(green)1 &\none[] &*(green)1&\shadehalf{green}{red}{1}&\shadehalf{green}{red}{1}&\none[]&\shadehalf{green}{red}{1} 
    \\
    \none[\scalebox{0.75}{\vdots}] & \none & \none & \none[{\ddots}]  & \none& \none \\
    \none[\scriptstyle i-1] & \none[]&   \none[]& \none &  &*(green)1 &*(green)1&\none[] & *(green)1&\shadehalf{green}{red}{1} &\shadehalf{green}{red}{1} &\none[]&\shadehalf{green}{red}{1} 
    \\
    \none[\scriptstyle i] & \none[]&   \none[]& \none & \none &*(green)1 &*(green)1&\none[] & *(green)1&\shadehalf{green}{red}{1} &\shadehalf{green}{red}{1} &\none[]&\shadehalf{green}{red}{1} 
    \\
    \none[\scriptstyle i+1] & \none[]&   \none[]& \none & \none &\none& &\none[] & &*(red)1 &*(red)1 &\none[]&*(red)1 
    \\
    \none[\scalebox{0.75}{\vdots}] & \none & \none & \none & \none &  \none & \none&\none[{\ddots}] \\
    \none[\scriptstyle j-1] & \none & \none & \none & \none & \none & \none& \none&&*(red) 1&*(red) 1&\none&*(red) 1\\
    \none[\scriptstyle j] & \none & \none & \none & \none & \none &  \none&\none&\none&*(red) 1&*(red) 1&\none&*(red) 1\\ 
    \none[\scriptstyle j+1] & \none & \none & \none & \none & \none &  \none&\none&\none&\none&&\none&\\  
    \none[\scalebox{0.75}{\vdots}] & \none & \none & \none &  \none& \none&\none & \none &  \none& \none&\none&\none[{\ddots}] \\
    \none[\scriptstyle n-1] & \none & \none & \none & \none & \none  & \none & \none&  \none&\none&\none&\none&\\     
    \end{ytableau}

    \hspace{1em}
    \begin{ytableau}
   \none[\scriptstyle 2] & \none[\scriptstyle 3] & \none[\scalebox{0.75}{\dots}] & \none[\scriptstyle i]  & \none[\scriptstyle i+1] 
    & \none[\scriptstyle i+2]&\none[\scalebox{0.75}{\dots}]
    &\none[\scriptstyle j]  & \none[\scriptstyle j+1]
     & \none[\scriptstyle j+2]&
    \none[\scalebox{0.75}{\dots}] & \none[\scriptstyle n] \\
   &  & \none &  &*(green)1 &*(green)1&\none[] &*(green)1 & &&\none[]& 
    \\
   \none[]&  & \none &  &*(green)1 &*(green)1 &\none[] &*(green)1& &&\none[]&
    \\
     \none & \none & \none & \none[{\ddots}] & \none \\
   \none[]&   \none[]& \none &  &*(green)1 &*(green)1&\none[] & *(green)1& & &\none[]&
    \\
   \none[]&   \none[]& \none & \none &*(green)1 &*(green)1&\none[] & *(green)1& &&\none[]&
    \\
     \none[]&   \none[]& \none & \none &\none& &\none[] & &*(green)1 &*(green)1 &\none[]&*(green)1 
    \\
   \none & \none & \none & \none[{\ddots}] & \none \\
    \none & \none & \none & \none & \none & \none& \none&&*(green) 1&*(green) 1&\none&*(green) 1\\
   \none & \none & \none & \none & \none &  \none&\none&\none&*(green) 1&*(green) 1&\none&*(green) 1\\ 
   \none & \none & \none & \none & \none &  \none&\none&\none&\none&&\none&\\  
    \none & \none & \none & \none[{\ddots}] & \none \\
 \none & \none & \none & \none & \none  & \none & \none&  \none&\none&\none&\none&\\     
    \end{ytableau}

    \end{tabular}%

    }     \caption{
    Illustrating the relation    $ H([i+1:j]) = H([1:j]) - H([1:i]) +  \sum_{k = 1}^i \sum_{\ell=i+1}^j 2e_ {k,\ell}$.
     \\
    On the left: \( H([1:i]) \) is shown in \textcolor{green}{green} and \( H([1:j]) \) in \textcolor{red}{red};
    on the right: $H([i+1:j])$.
}
\label{fig: intervals}
    \centering 
    \end{figure}
    \end{mdframed}
Therefore, applying \eqref{eq: vkl izkl}, $H(I) = H(\cup I_i) = \sum_i H(I_i) -\sum_{k,\ell} 2 e_{k,\ell}$, which shows
    that the vectors from the theorem statement generate $\Lat \Im H$. To see they are linearly independent,
    we once again refer to Figure \ref{fig: H123}, from which we see that, $H[1:n-1], H[1:n-2],
\dots, H[1:1], 2e_{n-2,n-1}, 2e_{n-3,n-1}, \dots, 2e_{1,2}$ form an upper triangular matrix
when we order the coordinates in the order $(n-1,n), (n-2,n), \dots,
(1,n),(n-2,n-1),(n-3,n-1), \dots,(1,2)$. Therefore, these vectors are linearly independent and this concludes the proof of first part the theorem.

For the second part, that is to prove that we can change a vector $2e_{i,j}$ with $i<j$ from the basis with $2 \cdot \one$, it is enough to observe that
   \[
   2 \cdot \one = 2 H([1:n-1]) + 2 \sum_{i<j<n}e_{i,j}.
   \]
\end{proof}

\section{Main results}
Our first result considers
the problem for symmetric Bernoulli random variables,
and illustrates in the neatest possible way the concepts of the proof for the most general case. In this simple context and in the specific case $n=3$, the problem has
also appeared in the French journal \cite{RMS-134-2-2024}.
\subsection{The binary case}

 \begin{theorem}
 \label{thm: binary} Let $\boldsymbol{X}_{\!1},  \dots, \boldsymbol{X}_{\!n}$ be $d$-dimensional vectors whose entries are iid samples from a
random variable $X\sim Ber(1/2)$. Let $p_d$ be the probability that these vectors are equidistant. Then, as $d \to \infty$,
 \begin{equation}
     \label{eq: main asymp binary}
 p_d \simi
\frac{1}{d^{(m-1)/2}}\cdot
\sqrt{\frac{2^{3m-2n-1}}{m\pi^{m-1}}},
\quad
\text{where $ m = \binom{n}{2}$}.
 \end{equation}
 \end{theorem}
 \begin{proof}
    Let us define
    \[
    \dijm{i}{j}{\ell} := \ind{\boldsymbol{X}_i^{(\ell)} \neq 
    \boldsymbol{X}_j^{\!(\ell)}},
\]
and denote the usual Hamming distance between binary vectors by
    \[d_H\lbrb{\boldsymbol{X}_i,\boldsymbol{X}_j} :=
    \sum_{\ell=1}^d
    \dijm{i}{j}{\ell}
    .
    \]
 Next, we quantify the fact that all $d_H(\boldsymbol{X}_i,\boldsymbol{X}_j)$ are equal for $i \neq j$ by the
    $m:=\binom{n}{2}$ equations $d_H(\boldsymbol{X}_i,\boldsymbol{X}_j) = d_H(\boldsymbol{X}_{1},\boldsymbol{X}_2)$. Note that
    if $(i,j)=(1,2)$ the equation is trivial, and we can thus omit it. Therefore
    $p_d$ is exactly equal to
    the probability of the event
  \[
\begin{cases}
    \sum_{\ell = 1}^d\left( \dijm{1}{3}{\ell} - \dijm{1}{2}{\ell} \right) = 0, \\
    \sum_{\ell = 1}^d \left( \dijm{1}{4}{\ell} - \dijm{1}{2}{\ell} \right) = 0, \\
    \vdots \\
    \sum_{\ell = 1}^d \left( \dijm{n-1}{n}{\ell} - \dijm{1}{2}{\ell} \right) = 0,
\end{cases}
\]
    which written in vector form with
    \[
\boldsymbol{\delta^{\!(\ell)}}
    :=\lbrb{\dijm{i}{j}{\ell}}_{1\leq i<j<n}= 
\begin{bmatrix}
  \dijm{1}{3}{\ell} \\
  \dijm{1}{4}{\ell} \\
  \vdots \\
  \dijm{n-1}{n}{\ell}
\end{bmatrix} \in \{0,1\}^{m-1}, 
\quad 
\one:=
\begin{bmatrix}
  1 \\
  1 \\
  \vdots \\
  1
  \end{bmatrix}
    \]
    becomes
    \begin{equation}
        \label{eq: sum=0}
        \sum_{\ell=1}^d \boldsymbol{V}_{\! \ell}:=
    \sum_{\ell=1}^{d}\lbrb{ \boldsymbol{\delta^{\!(\ell)}}
    - \boldsymbol{1}\cdot\dijm{1}{2}{\ell}} = \mathbf{0}.
    \end{equation}
   The vectors $\boldsymbol{V_i}$
   are iid with common law $\boldsymbol{V}$. We will estimate the probability of \eqref{eq: sum=0} by the local multidimensional CLT. Using the theorem from \cite[p. 3]{Krafft-1967}, we obtain
\begin{equation}
\label{thm: Krafft}
\P \lbrb{ \sum_{\ell=1}^d \boldsymbol{V}_{\!\ell} = \mathbf{0} }
\simi 
\frac{\labsrabs{\Lat \boldsymbol{V}}}{\sqrt{(2\pi d)^{m-1}\,|\Var(\boldsymbol{V})|}},
\end{equation}
where $|\Lat \boldsymbol{V}|$ denotes the determinant of a matrix formed by a lattice basis of the lattice generated by the support of $\boldsymbol{V}$ (see Section~\ref{sec: Lattice}), and $|\Var(\boldsymbol{V})|$ is the determinant of the covariance matrix of $\boldsymbol{V}$.

    Therefore our problem is
    reduced to understanding the covariance matrix of $\boldsymbol{V}$ and the lattice in $\Zb^{m-1}$ spanned by
    its support. The first task is straightforward:
    we calculate that $\Ebb{\delta_{i,j}} = 
    \Ebb{\lbrb{\delta_{i,j}}^2}
    = 1/2$, and $\dij{i_1}{j_1}\indep\dij{i_2}{j_2}$
    if $(i_1,j_1)\neq(i_2,j_2)$, which leads to
    \[
    \Cov(\dij{i_1}{j_1} - \delta_{n-2,n},
    \dij{i_2,j_2}-\delta_{n-2,n}) = \frac{1}{4} + \frac{1}{4}\ind{(i_1,j_1) = (i_2,j_2)}.
    \]
    The determinant of this matrix is $4^{-(m-1)}m$, see Proposition \ref{prop: pachi krak}. 
    
    We are left with the more subtle problem 
    of finding  a basis of the induced  by $\Im \boldsymbol{V}$ 
  lattice, which we do by using the
  techniques developed in Section \ref{sec: Lattice}, in particular Theorem \ref{thm: lattice}.
  \begin{lemma}\label{lem: H^ast}
  Define $H^\ast:\{0,1\}^n\to \Zb^{m-1}$ to be the transformation from \eqref{eq: sum=0} such that
  $\boldsymbol{V} = H^\ast(\boldsymbol{X})$, that is,
  \[
  H^\ast(\boldsymbol{X}) = \boldsymbol{\delta} - \boldsymbol{1} \cdot \delta_{1,2}.
  \]
      Then then vectors \begin{itemize}
          \item $H^\ast([1,0,0,\dots,0, 0])$,
      $H^\ast([1,1,0,\dots,0 ,0]),
      \dots,
      H^\ast([1,1,1,\dots,1,0])$
      \item alongside these of the form $2e_{i, j}$ for all $i<j$ such that $2<j<n$,
      \end{itemize}
      form a basis of $\Lat \boldsymbol{V} :=
      \Lat \supp \boldsymbol{V}$. Moreover, its fundamental volume is
      $\labsrabs{\Lat \boldsymbol{V}} = 2^{m-k}$.
  \end{lemma}
  \begin{proof}[Proof of Lemma \ref{lem: H^ast}]
  We will use the formalism of Section \ref{sec: Lattice}, and in particular Theorem 
\ref{thm: Krafft}. However, as $H(\boldsymbol{X}) = (\delta_{i,j})_{1\leq i<j\leq n}$  is $m$-dimensional, it is natural to embed $H^\ast$ into $\Rb^m$, which we do by adding a zero as a $(1,2)$ coordinate. Therefore, denoting the embedded version by $H_e^\ast$, we have that 
\begin{equation}
\label{def: embed}
H_e^\ast(\boldsymbol{X}) = H(\boldsymbol{X}) - \boldsymbol{1} \cdot \delta_{1,2}(\boldsymbol{X}),
\end{equation}
where $\delta_{1,2}(\boldsymbol{X})$ is the coordinate of $H(\boldsymbol{X})$ along $e_{1,2}$. We show that a basis of $\Lat H_e^\ast:=\Lat \Im H_e^\ast$ is obtained by applying $H_e^\ast$
to a basis of $\Lat H$ from Theorem \ref{thm: lattice}.

We recall that from Theorem \ref{thm: lattice}, the vectors
\begin{itemize}
    \item $H([1:1]), \dots, H([1:n-1])$,
    \item all $2e_{i,j}$ with $j<n$
    and $(i,j)\neq(1,2)$,
    \item and $2 \cdot \one$
\end{itemize}
form a basis of $\Lat H$. What happens if we subtract the $(1,2)$th coordinate of these vectors?
\begin{itemize}
    \item The transformation of the vectors from Figure \ref{fig: intervals}, $H([1:1]), \dots, H([1:n-1])$, is shown in Figure \ref{fig: H transformed},
    \item the vectors $2e_{i,j}$ with $2<j<n$ do not change,
    \item and $2 \cdot \one$ is transformed to $\mathbf{0}$. This is due to our embedding procedure: $\boldsymbol{V} = H(\boldsymbol{X})$ is $m-1$ dimensional, and $H(\boldsymbol{X})$ is $m$-dimensional.
\end{itemize}

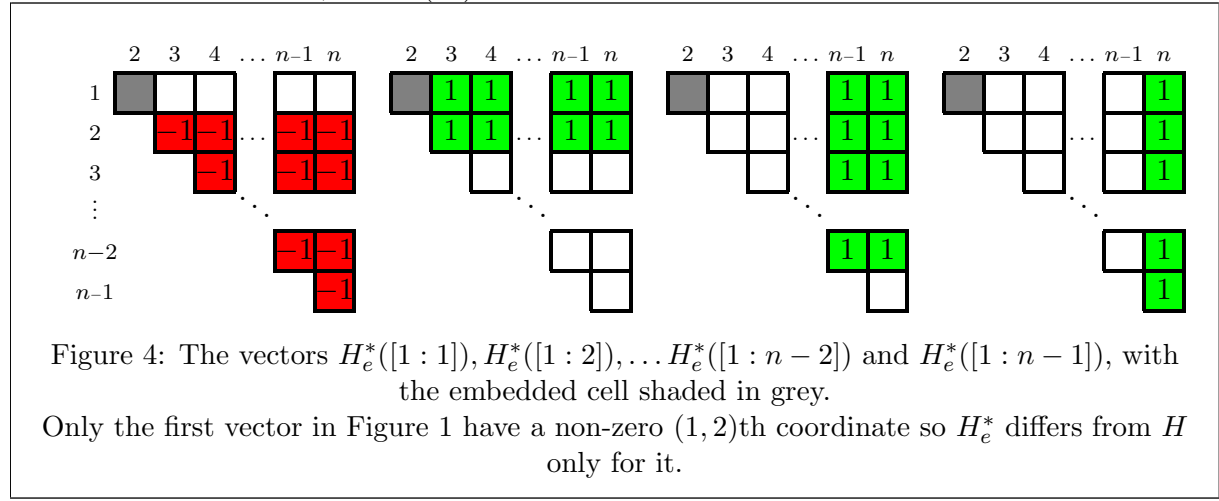
\begin{mdframed}[linecolor=black, linewidth=0.5pt, roundcorner=5pt]
    \begin{figure}[H]
    \centering
    \scalebox{1.05}{%
    \begin{tabular}{r@{}l}
    \ytableausetup {mathmode, boxframe=0.1em, boxsize=1.2em}  

    \begin{ytableau}
    \none & \none[\scriptstyle 2] & \none[\scriptstyle 3] & \none[\scriptstyle 4] & \none[\scalebox{0.75}{\dots}] & \none[\scriptstyle n\scalebox{0.5}{$-$}1] & \none[\scriptstyle n] \\
    \none[\scriptstyle 1] & *(gray) &  &  & \none & & \\
    \none[\scriptstyle 2] & \none &  *(red)-1 &  *(red)-1 & \none[\scalebox{0.75}{\dots}] & *(red)-1 &  *(red)-1 \\
    \none[\scriptstyle 3] & \none & \none &  *(red)-1 & \none &  *(red)-1 &  *(red)-1 \\
    \none[\scalebox{0.75}{\vdots}] & \none & \none & \none & \none[{\ddots}] & \none \\
    \none[\scriptstyle n-2] & \none & \none & \none & \none & *(red)-1  &   *(red)-1\\
    \none[\scriptstyle n\scalebox{0.5}{$-$}1]  & \none & \none & \none & \none & \none & *(red)-1  \\        
    \end{ytableau}

    \hspace{0.5em}

    \begin{ytableau} 
    \none[\scriptstyle 2] & \none[\scriptstyle 3] & \none[\scriptstyle 4] & \none[\scalebox{0.75}{\dots}] & \none[\scriptstyle n\scalebox{0.5}{$-$}1]  & \none[\scriptstyle n] \\
       *(gray)
    & *(green)1 & *(green)1 & \none & *(green)1 & *(green)1 \\
    \none & *(green)1 & *(green)1 & \none[\scalebox{0.75}{\dots}] & *(green)1 & *(green)1 \\
    \none & \none &  & \none &  &  \\
    \none & \none & \none & \none[{\ddots}] & \none \\
    \none& \none & \none & \none &  &  \\
    \none& \none & \none & \none & \none &  \\
    \end{ytableau}

    \hspace{0.5em}
    \begin{ytableau}
    \none[\scriptstyle 2] & \none[\scriptstyle 3] & \none[\scriptstyle 4] & \none[\scalebox{0.75}{\dots}] & \none[\scriptstyle n\scalebox{0.5}{$-$}1]  & \none[\scriptstyle n] \\
    *(gray)  &  &  & \none & *(green)1 &*(green)1   \\
    \none &  &   & \none[\scalebox{0.75}{\dots}]&*(green)1  &*(green)1  \\
    \none & \none &  & \none & *(green)1  &  *(green)1 \\
    \none & \none & \none & \none[{\ddots}] & \none \\
     \none& \none & \none & \none & *(green)1  &*(green)1   \\
    \none& \none & \none & \none & \none &   \\        
    \end{ytableau}

    \hspace{0.5em}
    \begin{ytableau}
    \none[\scriptstyle 2] & \none[\scriptstyle 3] & \none[\scriptstyle 4] & \none[\scalebox{0.75}{\dots}] & \none[\scriptstyle n\scalebox{0.5}{$-$}1] & \none[\scriptstyle n] \\
    *(gray)  &  &  & \none &   &*(green)1   \\
    \none &  &   & \none[\scalebox{0.75}{\dots}] &  &*(green)1   \\
    \none & \none &   & \none &   &*(green)1    \\
    \none & \none & \none & \none[{\ddots}] & \none \\
    \none& \none & \none & \none &   &*(green)1    \\
    \none& \none & \none & \none & \none &  *(green)1 \\        
    \end{ytableau}
    \end{tabular}%
    } 
    \caption{The vectors $H^\ast_e([1:1]), H^\ast_e([1:2]), \dots H^\ast_e([1:n-2])$ and $H^\ast_e([1:n-1])$, with the embedded cell shaded in grey.\\
    Only the first vector in Figure \ref{fig: H123} have a non-zero $(1,2)$th coordinate
    so $H^\ast_e $ differs from $H$ only for it. }
    \label{fig: H transformed}
    \end{figure}
        
\end{mdframed}

We argue that the vectors from Figure \ref{fig: H transformed} alongside $2e_{i,j}$ for $2<j<n$ and 
$(i,j)\neq (1,2)$
are a basis of $\Lat H^\ast_e$. This is true
 because first
\begin{itemize}
    \item $H^\ast_e([1:2]) = H([1:2]), H^\ast_e([1:3])=H([1:3]), \dots, H^\ast_e([1:n-1]) =
    H([1:n-1])$, and $2e_{i,j}$ for $2<j<n$
    and $(i,j)\neq(1,2)$ are independent by Theorem \ref{thm: lattice}
    \item $  H^\ast_e([1:1])$ is independent of them because it has coordinates $(0,-1)$ along $(e_{1,n},e_{2,n})$, and all other vectors have equal coordinates along these them.
\end{itemize}

Further, we have that $\dim \Lat H^\ast_e$ is at most the dimension
of the space, that is $m$ minus one, because of the
null coordinates along $e_{1,2}$. The chosen vectors vectors are
exactly $m-1$, because they are one less than the vectors
in the basis of $\Lat H$, which are $m$. Moreover they are in $\Im H^\ast$ which is clear from the definition \eqref{def: embed}. Therefore
they form a basis of $\Lat H_e^\ast$. 

The reverse embedding, that is removing coordinate $(1,2)$, produces a lattice of $H^\ast$. It remains to calculate the determinant, which we do by ordering the vectors in a suitable block structure:
 let $M$ be an $(m-1)\times (m-1)$ matrix of $0$s and $1$s. We enumerate the rows and columns of $M$ by $(j,i)$ with $i<j$ and $(i,j)\neq (1,2)$, and order them in the reverse lexicographical order. In this notation the rows $M_{(j,i)}$ of $M$ are
\begin{equation*}
M_{(j,i)}=2e_{i,j} \text{ if } i<j<n, \quad
M_{(n,i)}=H([1:i+1]) \text{ for } i<n-1,
\quad
\text{and}
\quad 
M_{(n,n-1)}=\one-H([1:1]).
\end{equation*}
Then $M$ is lower diagonal, and can be written in the form
\begin{equation}
\label{eq:matrix M}
M :=
\begin{bNiceArray}{c|c}[margin]
 2I_{(m-n)\times(m-n)} & \boldsymbol{0} \\
 \hline
 \ast                & M_1
\end{bNiceArray}
\end{equation}
where $I$ is the identity matrix, and  $M_1$ is lower diagonal $(n-1)\times (n-1)$ matrix with 1s along the diagonal.
 Thus, the determinant is $2^{m-n}$ as claimed.

  \end{proof}
Equipped with the last lemma, and because $\boldsymbol{V}$ is
$m-1$ dimensional,
equation \eqref{thm: Krafft} implies
that
\[
\P\lbrb{\sum_{\ell=1}^d
\boldsymbol{V}_{\!\ell}= \mathbf{0}} \simi \frac{2^{m-n}}{\sqrt{(2\pi d)^{m-1}4^{-(m-1)}m}} = \frac{1}{d^{(m-1)/2}}\cdot
\sqrt{\frac{2^{3m-2n-1}}{m\pi^{m-1}}},
\quad
\text{where}
\quad
m = \binom{n}{2},
\]
as claimed.
\end{proof}
\subsection{The lattice case}
The arguments from the binary case in fact
can be used for
all random vectors
with integer
support. Although
this is clear for
the idea of the local
CLT, the rather surprising fact is that
the basis from Theorem \ref{thm: lattice} obtained in the context of binary vectors 
is
fundamental in the general case as well. Another gratifying fact is that the exact constant in the asymptotics can be calculated for any family of vectors considered.


\begin{theorem}
    \label{thm: general lattice}
 Let $\boldsymbol{X}_1,  \dots, \boldsymbol{X}_n$ be $d$-dimensional vectors
 whose entries are iid
 samples from a random variable $X$
 with finite lattice support. Let 
$p_d$ be the probability that these vectors are equidistant.
Then,  for iid $X_1,X_2,X_3\sim X$,
and constants
 \[
m = \binom{n}{2},
\quad
\text{and}
\quad
C_1 :=  \mathrm{Cov}((X_1 - X_2)^2, (X_1-X_3)^2) = \Var\lbrb{\lbrb{X-\Ebb{X}}^2},
\]
we have that as $d \to \infty$,
 \begin{equation}
     \label{eq: main asymp lattice}
 p_d \simi
\frac{1}{d^{(m-1)/2}}\cdot
\Big(
m(2\pi)^{m-1}(\Var X)^{2(m-n)}(4\lbrb{\Var X}^2+(n-2)C_1)^{n-1}
\Big)^{-1/2}.
 \end{equation}
\end{theorem}
\begin{remark} For specific calculations,
 we note that in terms of the moments $m_i:=\Ebb{X^i}$ and 
 central moments
 $\mu_i :=\Ebb{(X-m_1)^i}$,
\[
C_1 
= m_4 - 4m_3m_1 - m_2^2 + 8m_1^2 m_2 - 4m_1^4= \mu_4 - \mu_2^2 = \Var\lbrb{(X-m_1)^2}.
\] 
\end{remark}
\begin{remark}\label{rem: metric}
The essence of the arguments in the following proof is not specific to the Euclidean distance. In particular, if one considers the $L^p$ distance, the objects of interest in the spirit of
\eqref{def: Vij} would be
 \[
 V_{i,j}:= 
 \lbrb{\boldsymbol{X}_{\! i}^{\!(\ell)}-\boldsymbol{X}_{\! j}^{\!(\ell)}}^p -\lbrb{\boldsymbol{X}_{\! 1}^{(\ell)}-\boldsymbol{X}_{\! 2}^{\!(\ell)}}^p.
 \]
Therefore, $|\Var(\boldsymbol{V})|$ remains unaffected and retains the form in~\eqref{eq: CovV}, expressed in terms of $C_0$ and $C_1$ from~\eqref{def: C0}. However, applying $\Lat\boldsymbol{V}$ can cause difficulties. To illustrate, suppose $X$ is integer-valued. If $p$ is a natural number (i.e., an integer), then $\supp \boldsymbol{V}$ also has integer entries. Otherwise, if $p$ is not a natural number (i.e., non-integral), the $L^p$-distance may produce irrational values. In this latter case, since $\boldsymbol{V}$ is no longer confined to a lattice, we must resort to the most general result, Theorem~\ref{thm: general}.

\end{remark}

\begin{proof}[Proof of Theorem]
We will follow the ideas from the proof in the binary case presented in Theorem \ref{thm: binary}.
In the same manner, we have the analogue of \eqref{thm: Krafft},
  \begin{equation}
        \label{eq: Krafft_general}
  p_d =   \P \lbrb{\sum_{\ell=1}^d \boldsymbol{V}_{\! \ell} = \mathbf{0}} \simi \frac{\labsrabs{\Lat \boldsymbol{V}}}{\sqrt{(2\pi d)^{m-1}|\Var(\boldsymbol{V})|}},
 \end{equation}
 where the vector $\boldsymbol{V}_{
 \!\ell}$ is composed not of the Hamming distances between
 coordinates, but of the Euclidean ones (or another one as mentioned in Remark \ref{rem: metric}). Explicitly, the coordinates of $\boldsymbol{V_\ell} = (V_{\ell}^{(i,j)})$ with
 $1\leq i < j\leq n$ and
 $(i,j)\neq (1,2)$, are
 \begin{equation*}
 V^{(i,j)}_{\ell} := 
 \lbrb{\boldsymbol{X}_{\! i}^{\!(\ell)}-\boldsymbol{X}_{\! j}^{\!(\ell)}}^2 -\lbrb{\boldsymbol{X}_{\! 1}^{(\ell)}-\boldsymbol{X}_{\! 2}^{\!(\ell)}}^2.  
 \end{equation*}
 We start by showing that the lattice spanned by 
  vectors from a general 
  discrete distribution is basically the same (modulo scaling) as the one generated by binary ones. First note that by multiplying $X$ by the  lcm of $\supp{X}$,
  we can assume that its elements are integer. Moreover,
  the same scaling argument shows that we can also assume that
    $\gcd \supp X$ is 1.

Let us 
  introduce a random variable with the same law as $\boldsymbol{V}_{\! i}$, constructed from  iid
  $X_1, \dots, X_n \sim X$,
  \begin{equation}
        \label{def: Vij}
\boldsymbol{V}:= (V_{i,j}) = ((X_i-X_j)^2 - (X_1-X_2)^2).
  \end{equation}
  To use our results from the binary case, let us consider the problem modulo 2. Using the transformation $H^\ast$ from Lemma \ref{lem: H^ast}, $\boldsymbol{V}\pmodsmall{2} = H^\ast(\boldsymbol{X}\pmodsmall{2})$,
  and because $\gcd(\supp{\boldsymbol{X}}) = 1$,
  the basis from Lemma \ref{lem: H^ast}
  describes all values of $\boldsymbol{V}\pmodsmall{2}$.
Moreover, let us recall the relation
\eqref{eq: Hi+Hj - Hij}
\[
    H\lbrb{\{i\}} + H(\{j\}) - H(\{i,j\}) = 2 e_{i, j},
\]
which, after substituting in the definition of the embedding $H_e^\ast$ of $H^\ast$ in $\R^m$, see \eqref{def: embed}, leads to
\begin{align*}
H^\ast_e(\{i\}) &+ H^\ast_e(\{j\})-H^\ast_e(\{i,j\}) = 2e_{i,j}
\quad
\text{for $(i,j)\neq(1,2)$}.
\end{align*}
Therefore, every vector from $\Lat \boldsymbol{V}$ can be expressed in terms of elements of $\Lat H^\ast$, or equivalently $\Lat\boldsymbol{V}\subseteq \Lat H^\ast $.
For the reverse inclusion, following the same idea, consider what happens 
if
$\boldsymbol{X} = x e_{i}$
for $x$ from the support of $X$.
Then, for $(i,j)\neq(1,2)$,
\begin{align*}
\boldsymbol{V}(xe_i) &+ 
\boldsymbol{V}(xe_j)-
\boldsymbol{V}(x(e_i+e_j))
= x^2(H^\ast_e(\{i\}) + H^\ast_e(\{j\})-H^\ast_e(\{i,j\}) = 2x^2e_{i,j}.
\end{align*}
Therefore $2x^2e_{i,j}\in \supp{\boldsymbol{V}}$ for each $x \in \supp X$.
Moreover, as $x^2 H^\ast([1,1,\dots,1,0,\dots,0])$ also is in $\supp\boldsymbol{V}$, we obtain that $x^2\Lat H^\ast\subseteq\Lat \boldsymbol{V}$ for each $x \in \supp{X}$. Therefore, by the
fact that the gcd of the latter values is $1$ and Bézout's identity, $\Lat H^\ast \subseteq \Lat \boldsymbol{V} $, which along the previous inclusion gives that the two lattices do in fact coincide,
and a fortiori, from Lemma \ref{lem: H^ast}
\begin{equation}
    \label{eq: det general lattice }
    \labsrabs{\Lat \boldsymbol{V}} =
    \det \Lat H^\ast = 2^{m-n}.
\end{equation}
Our next step is to calculate
$\labsrabs{\Var(\boldsymbol{V})}$. For brevity, we will denote
\[
\alpha :=(i,j), \quad \beta = (k,\ell),
\quad
\text{and}
\quad
A_{\alpha} =A_{i,j}:= (X_i-X_j)^2.
\]
Then $\labsrabs{\Var(\boldsymbol{V})}$ is  the determinant of the matrix composed of 
\begin{equation}\label{eq: Cov(V)}
\begin{split}
\Cov(V_{\alpha},V_{\beta}) &= \Cov(A_{\alpha}- A_{1,2}, A_{\beta} -A
_{1,2})\\
&=\Cov(A_{\alpha}, A_{\beta}) - \Cov(A_{1,2}, A_{\alpha}) 
-\Cov(A_{1,2}, A_{\beta}) 
+\Cov(A_{1,2}, A_{1,2}).
\end{split}
\end{equation}
Since $X_i$ are iid, everything may be expressed in terms of the constants
\begin{equation}  
\label{def: C0}
C_0:= \Var(A_{1,2}) 
\quad
\text{and}
\quad
C_1:= \Cov(A_{1,2},A_{1,3}),
\end{equation}

with differences on the number of common elements between $\{i,j\}$ and $\{k,\ell\}$, and it is thus natural to recall the adjacency matrix $\cH$
of the line graph of the complete $n$-graph
$H$, which we introduced in 
\eqref{def: cH}. Its elements are
\[
h_{\alpha, \beta} = \begin{cases}
    1, & \text{if $|\alpha \cap \beta| = 1$};\\
    0, & \text{otherwise}.
\end{cases}
\]
 Therefore, equation \eqref{eq: Cov(V)} can be written as
\begin{align*}
\Cov(V_\alpha, V_\beta) =C_0\one_{\{\alpha = \beta\}} +C_1 h_{\alpha,\beta}- C_1 h_{(1,2),\alpha} - C_1h_{(1,2),\beta} + C_0.
\end{align*}

To put the last in matrix form, let $J$ be the $m\times m$ matrix, which have 1s in all columns except $(1,2)$, and $J_{1,2}$ be the matrix with 1s in the rows $\alpha$ with $|\alpha \cap \{1,2\}| = 1$ and 0s in all other entries, or visually
\begin{equation}
    \label{eq:J, J^-}
J := \begin{bNiceArray}{c|w{c}{4.5em}|w{c}{7.5em}}[margin]
  0 & {\one_{1\times(2n-4)}} &{\one_{1
\times (m-(2n-4)-1)}}  \\ 
  \hline
  {\mathbf{0}} &
  {\one} &{\one}  \\
  \hline
  {\mathbf{0}} &
  {\one} &{\one} 
\end{bNiceArray},
\quad
\text{and}
\quad
J_{1,2}:=\begin{bNiceArray}{c|w{c}{0.5em}|c}[margin]
  0 &{\mathbf{0}} &{\mathbf{0}}  \\ 
  \hline
  {\one}  &
  {\one} &{\one}  \\
  \hline
  {\mathbf{0}} &
   {\mathbf{0}}  & {\mathbf{0}}  
\end{bNiceArray}.
\end{equation}
Finally, for a matrix $M$ of dimension $m\times m$, we set $M^-$ to be $M$ after the removal of row $(1,2)$ and column $(1,2)$. With this notation, the above equality for individual covariances translates to
\begin{equation}
    \label{eq: Cov Matrix}
    \begin{split}
\Cov(\boldsymbol{V} ) &= C_0 I^- +C_1 
\cH^- -C_1J_{1,2}^--C_1(J_{1,2}^t)^-+C_0 J^{-} 
\\
&= C_0(I^{-}+J^-) +C_1(\cH^{-}-J^-_{1,2}-
J^t_{1,2}).
 \end{split}
\end{equation}
To work with $\mathcal{H}$ itself, and not $\mathcal{H}^-$, let us define the matrix
\begin{equation}
    \label{eq: detSigma}
\mathcal{A} := \begin{bNiceArray}{c|c}[margin]
  C_0 & C_0\one_{1
  \times (m-1)}  \\ 
  \hline
  \mathbf{0}_{(m-1)\times 1} & \text{Cov}(\boldsymbol{V})
\end{bNiceArray},
\quad
\text{so}
\quad
\det \cA = C_0  \labsrabs{\Var(\boldsymbol{V})}.
\end{equation}
Then, with the block structure of
\eqref{eq:J, J^-},
\[
\Hc = 
\begin{bNiceArray}{c|c|c}[margin]
  0            & \mathbf{0}   & \mathbf{0} \\ 
  \cline{1-3} 
  \one         & \Block{2-2}{\Hc^-} \\
  \cline{1-1} 
  \mathbf{0}   &              &              \\
\end{bNiceArray},
\]
because the second row describes exactly coordinates where $H_{1,2}$ is non-zero, so equation \eqref{eq: Cov Matrix} gives that
\begin{equation}\label{eq: A Matrix}\mathcal{A} = 
C_0\lbrb{ I +
\begin{bNiceArray}{c|w{c}{0.5em}|c}[margin]
  0 & {\one} &{\one}  \\ 
  \hline
  {\mathbf{0}} &
  {\one} &{\one}  \\
  \hline
  {\mathbf{0}} &
  {\one} &{\one} 
\end{bNiceArray}
}
+ C_1\lbrb{\cH
-
\begin{bNiceArray}{c|w{c}{0.5em}|c}[margin]
  0 & {\boldsymbol{0}} &{\mathbf{0}}  \\ 
  \hline
  {\one} &
  {\one} &{\one}  \\
  \hline
  {\mathbf{0}} &
  {\boldsymbol{0}} &{\mathbf{0}} 
\end{bNiceArray}
- 
\begin{bNiceArray}{c|w{c}{0.5em}|c}[margin]
  0 & {\one} &{\mathbf{0}}  \\ 
  \hline
  {\boldsymbol{0}} &
  {\one} &{\boldsymbol{0}}  \\
  \hline
  {\mathbf{0}} &
  {\one} &{\mathbf{0}} 
\end{bNiceArray}
}
.
\end{equation}
We calculate $\det \cA$ by writing the matrix $\cA$ in a suitable basis, which is connected to the eigenstructure of $\cH$, described in Proposition \ref{prop: line Kn}. In a nutshell, Proposition \ref{prop: line Kn} 
provides that the vectors $\one$ 
and $H_i:=H(\{i\})$ for $i>1$ span  
a particular eigenspace of $\cH$. 
Next, because of our embedding,
see \eqref{eq: detSigma}, $e_{1,2}$ plays a special role,
and we are thus led to consider the set
\[
\mathcal{B}^\ast := 
\lbcurlyrbcurly{\one, e_{1,2}, H_2, H_3, \dots, H_n}.
\]
Note that,
\begin{equation}
    \label{eq: 2 one}
    2 \cdot \one = \sum_{i=1}^n H_i,
\end{equation}
because only $H_i$ and $H_j$ have non-zero coordinate along $e_{i,j}$ and
it is exactly 1. Moreover, from \eqref{eq: Hi+Hj - Hij},
\[\cH_{1,2}=H(\{1,2\}) = H_1 + H_2 - 2e_{1,2},\] the linear span of 
$\mathcal{B}^\ast$ coincides with the linear span of
\[
\mathcal{B} := 
\lbcurlyrbcurly{\one, \cH_{1,2}, H_2, H_3, \dots, H_n},
\]
which turns out to be more suitable for calculations.

First let us ensure that the vectors in $\mathcal {B}$ are linearly independent. 
Consider a linear combination such that
\[
 \alpha \one + \beta \cH_{1,2} + \sum_{i>1}\alpha_iH_i = \boldsymbol{0}.
\]
Considering specific coordinates, we get
\begin{align*}
    \text{along $e_{1,2}$}&:
    \alpha + \alpha_2 = 0\\
    \text{for $i>2$, along $e_{1,i}$}&: 
    \alpha + \beta + \alpha_i= 0
    \quad\qquad \,\,\hspace{0.1em}(\text{so $\alpha_i$ are equal for $i>2$)} 
    \\
    \text{along $e_{2,3}$}&:
    \alpha  + \beta + \alpha_2 +\alpha_3= 0
    \quad (\text{so $\alpha_2= \alpha=0$)} \\
     \text{for $i,j>2$, along $e_{i,j}$}&:
    \alpha  + \alpha_i +\alpha_j= 0,
 \end{align*}
which implies that all coefficients are zero, so the considered linear combination is trivial. We note the argument is correct if $n>3$. For $n=3$, $\labsrabs{\Var(\boldsymbol{V})}$  can be calculated by hand and matches the expression for $n>3$
from \eqref{eq: CovV}.

 We can further see that the orthogonal complement of $\cB$, which we denote by $\mathcal{B}^\perp$, consists of eigenvectors with eigenvalue $-2$ of $\cH$. 
Indeed, take a vector $v \in \mathcal{B}^\perp$. By construction it is orthogonal to $\one$,  $\cH_{1,2}$, and $H_i$ for $i>1$,  Further, again recalling that 
$2\cdot\one = \sum_{i=1}^nH_i$, $v$ is
also orthogonal to $H_1$, and since $\cH_{1,2} = H_1 + H_2 - 2e_{1,2}$, we get that $\langle v, e_{1,2}\rangle = 0$.

Therefore, denoting by $\cH_{i,j}$ the $(i,j)$th column of $\cH$,
\[
\langle \cH_{i,j}, v\rangle   = \langle H_i + H_j - 2e_{i,j}, v\rangle = \langle - 2e_{i,j}, v\rangle = -2 v_{i,j},
\]
which means that $\cH v = -2 v$, so substituting in 
\eqref{eq: Cov Matrix},
\[
\cA v = C_0 \lbrb{v + \langlerangle{v, \one -e_{1,2}}}
 + C_1 \lbrb{-2v-\langlerangle{v, \cH_{1,2}}\one - \langlerangle{v,\one}\cH_{1,2}}
 = (C_0 - 2C_1)v.
\]
Thus to calculate $\det \cA$, it remains to understand
the block corresponding to $\mathcal{B}$.

In Proposition \ref{prop: det cA} from the Appendix, we directly calculate the image of $\mathcal{B}$ under $\cA$ to get  explicitly this block, and then,
by exploiting its arrowhead structure, to compute that its determinant is equal to
\[
mC_0\lbrb{C_0-2C_1}^{m-n}\lbrb{C_0 + C_1(n-4)}^{n-1} .
\]
Accounting for the vectors in $\mathcal{B}^\perp$, whose dimension is $m-(n+1)$, we get that
\begin{equation}\label{eq: CovV}
\det \cA = mC_0 (C_0-2C_1)^{m-n}(C_0 + C_1(n-4))^{n-1}
\eqinfo{\eqref{eq: detSigma}} C_0 \labsrabs{\Var(\boldsymbol{V})},
\end{equation}
which provides $\labsrabs{\Var(\boldsymbol{V})}$. To obtain 
the expression from the statement of the theorem, see \eqref{eq: main asymp lattice}, a direct calculation in terms of $m_i:=\Ebb{X^i}$ shows that
 \[
C_0
= 2m_4 - 8m_3m_1 + 2m_2^2 + 8m_1^2 m_2 - 4m_1^4,
\quad \text{and}
\]
\[
C_1 
= m_4 - 4m_3m_1 - m_2^2 + 8m_1^2 m_2 - 4m_1^4,
\] 
so $C_0-2C_1 = 4(\Var X)^2$. Substituting along \eqref{eq: det general lattice } in \eqref{eq: Krafft_general}, we get that
\[
p_d \simi
 \frac{\labsrabs{\Lat \boldsymbol{V}}}{\sqrt{(2\pi d)^{m-1}\labsrabs{\Var(\boldsymbol{V})}}}
 =
\frac{2^{m-n}}{\sqrt{(2\pi d)^{m-1}m(4\Var X)^{2(m-n)}(4(\Var X)^2+C_1(n-2))^{n-1}}},
\]
which after simplification gives the desired \eqref{eq: main asymp lattice}.
\end{proof}

\subsection{The general case}\label{subsec: the general case}
Consider now the case where $X$ follows a uniform law over $\{0, 1,\sqrt{2},\sqrt{3}\}$. Again central for the problem is the behaviour of the vector constructed from  iid
  $X_1, \dots, X_n \sim X$
\[
\boldsymbol{V}:= (V_{i,j}) = ((X_i-X_j)^2 - (X_1-X_2)^2).
\]
  However, $\supp (X_i-X_j)^2 = \{0, 1,2, 3,3-2\sqrt{2}, 4-2\sqrt{3}, 5-2\sqrt{6}\}$ does not span a proper lattice. To address the latter, we may consider these real numbers as vectors with rational coordinates in a suitable Hamel basis of the linear span of $\supp (X_i-X_j)^2$
  over $\Qb$: for example in our example we may pick $1, \sqrt{2}, \sqrt{3}$, and $\sqrt{6}$ as a basis, and then identify $3-2\sqrt{2}$ with the (transposed) vector $(3,-2,0,0)$. 
  
  When we dealt with integer lattices,
  their fundamental volumes were calculated on the usual natural scale: that is the determinant of one of their bases written in terms of the standard basis of $\Qb^n$. In a similar manner, one may see that in our context, we will instead need to fix a basis of 
  reference: in this section  $\mathcal{X} =\{x_0=0,x_1,x_2,\dots,x_k\}$ will be a set of real numbers, whose motivation is in fact to describe $\supp X$. 
  Define the sets
 Let
\[
\mathcal{X}_1 = \{ 2x_i x_j \mid 1 \le i,j \le k \} \quad \text{and} \quad \mathcal{X}_2 = \mathcal{X}_1 \cup \{ x_i^2 \mid 0 \le i \le k \}.
\]
We denote by
\[
\mathcal{L} = \Span_{\mathbb{Q}}(\mathcal{X}_1)
\]
the vector space over \(\mathbb{Q}\) generated by \(\mathcal{X}_1\), and by $
\ell = \dim_{\mathbb{Q}}(\mathcal{L})
$
its dimension.
Note that $\Span_\Qb(\Xc_1) = \Span_\Qb(\Xc_2)$,
so we may choose a single reference basis for both. To this end, fix an isomorphism
$
\varphi \colon {\cal L} \to \mathbb{Q}^{\ell},
$
and define the reference basis \({\cal B}_{\text{ref}}\) of \({\cal L}\) as the pull-back via \(\varphi\) of the standard basis of \(\mathbb{Q}^{\ell}\).
With respect to this basis, let ${\cal B}_1$ be a basis for $\Lat \mathcal{X}_1$ and ${\cal B}_2$ be a basis for $\Lat \mathcal{X}_2$. All fundamental volumes in this section will be calculated on the scale of $\Bc_{ref}$, in the sense $|\Lat \Xc_i| = \det \Bc_i$.

In the next lemma, which is the generalisation of Theorem \ref{thm: lattice},
  we describe a lattice basis of $\Lat \boldsymbol{V}$ but considering the latter
   not as $(m-1)$-dimensional  with $m=\binom{n}{2}$ as before, but
  as $(m-1) \times \dim \supp(X_i-X_j)^2$-dimensional. 

 \begin{lemma}\label{lem: lattice_r}
Using the setup from the beginning of this section, let $H_X:\mathcal{X}^n\rightarrow \mathbb{R}^{m}$ be the function:
 \begin{equation*}
H_\mathcal{X}(\boldsymbol{x})=\sum_{i<j}(x_i-x_j)^2 e_{i,j}.
 \end{equation*}
Then the vectors 
\begin{itemize}
\item $u\otimes H([1:i])$, for $u \in \Bc_2$ and $i<n$,
 \item $v\otimes e_{i,j}$, for $v \in \Bc_1$ and $i<j<n$,
\end{itemize}
form  a basis of  $\Lat H_\mathcal{X} := \Lat \Im H_\mathcal{X}$ considered as a lattice on ${\cal L}\otimes \mathbb{Q}^m$.
 \end{lemma}
\begin{proof}[Proof of Lemma \ref{lem: lattice_r}]
We start by introducing a suitable notation for describing vectors:
for a subset $I\subseteq [1:n]$ and $x_i, y \in \mathcal{X}$, denote by
\[
\lbrb{\lbrb{x_i}_{{i}\in I}; y}
\text{ the vector with entries $x_i$ at positions $i \in I$ and entries $y$ at all other positions.}
\]
Then in the style of calculations involving the overlapping map $H$,
one can directly check that, for $x,y \in \mathcal{X}$,
\begin{align*}
H_\mathcal{X}\lbrb{(x)_{\{i\}};0} +
H_\mathcal{X}\lbrb{(y)_{\{j\}};0} -
H_\mathcal{X}\lbrb{(x,y)_{\{i,j\}};0}
= 2xy e_{i,j}.
\end{align*}
Similarly, $H_\mathcal{X}\lbrb{(x)_{[1:i]};y} = (x-y)^2 H[1:i]$, so 
\[
H_\mathcal{X}\lbrb{(x)_{[1:i]};0} = x^2H([1:i]),
\quad\text{and} \quad
H_\mathcal{X}\lbrb{(x)_{[1:i]};0}
+ H_\mathcal{X}\lbrb{(y)_{[1:i]};0} - 
 H_\mathcal{X}\lbrb{(x)_{[1:i]};y} = 2xy H[1:i].
\]
Therefore the vectors from the theorem statement are in $\Lat H_\mathcal{X}.$
To see that they generate the entire lattice $\Lat H_\mathcal{X}$, by equation
\eqref{eq: interval}, we have
\begin{align*}
H_\mathcal{X}\lbrb{\boldsymbol{x}}
= \sum_{i<j}(x_i - x_j)^2e_{i,j} &= 
\sum_{i = 1}^n x_i^2 H\lbrb{\{i\}}
-\sum_{i<j}2x_ix_je_{i,j}\\
&=\sum_{i=1}^n x_i^2\lbrb{H[1:i] - H[1:i-1] + 2\sum_{k=1}^{i-1}e_{k,i}}+\sum_{i<j}2x_ix_je_{i,j}.
\end{align*}

To conclude the proof, it remains to be shown that the chosen vectors are linearly independent. However, we have already noted that $\Span_\Qb(\Xc_1) = \Span_\Qb(\Xc_2)=\mathcal{L}$,
and as we already know from Theorem \ref{thm: binary} that $H([1:i])$ for $i<n$ along with $e_{i,j}$ for $i<j<n$ are linearly independent in $\mathbb{Q}^m$, the results follows.
\end{proof}
Having described $\Lat H_\Xc$, we transfer this understanding to $\Lat \boldsymbol{V}$ as we did
in the transition from Theorem \ref{thm: lattice} to Lemma \ref{lem: H^ast}.
\begin{corollary}\label{cor:volumes}
Let $\mathcal{X}$, $\mathcal{X}_1$, $\mathcal{X}_2$ and ${\cal L}$ be as above. 
Denote by $H_\mathcal{X}^\ast:\mathcal{X}^n \rightarrow \mathbb{R}^{m-1}$ the mapping
\begin{equation*}
H_\mathcal{X}^*(\boldsymbol{x}) = \sum_{i<j} ((x_i-x_j)^2-(x_1-x_2)^2) e_{i,j}.
\end{equation*}
Then the set $\cB^\ast$ of vectors
\begin{itemize}
\item $u\otimes H^\ast([1:i])$, for $u \in \Bc_2$ and $i<n$,
 \item $v\otimes e_{i,j}$, for $v \in \Bc_1$ and for all $i<j$ such that $2<j<n$,
\end{itemize}
forms  a basis of  $\Lat \Im H^\ast_\mathcal{X}$. Moreover,
the fundamental volume of the lattice $\Lat \Im H^*_X$ is
\begin{equation*}
|\Lat \Im H^*_X| = 
(\det {\cal B}_1)^{m-n}(\det {\cal B}_2)^{n-1} = |\Lat \Bc_1|^{m-n}\cdot 
|\Lat \Bc_2|^{n-1}.
\end{equation*}
\end{corollary}
\begin{proof}[Proof of Corollary \ref{cor:volumes}]
The arguments are in essence the same as those used to generalise Theorem
\ref{thm: lattice} to Lemma \ref{lem: lattice_r}: from Lemma \ref{lem: H^ast}
we know that $H^\ast([1:i])$ and $2e_{i,j}$ form a basis of $\Lat H^\ast$ and note that $H^\ast(\{i\}) = H(\{i\})$ for
for $i>2$ and $H^\ast(\{i\}) = H(\{i\})-\one$ for $i=1,2$. Therefore
\begin{align*}
H^\ast_\mathcal{X}(\boldsymbol{x}) &=
\sum_{i=1}^n x_i^2 H(\{i\}) - \sum_{i<j}2x_ix_je_{i,j}
-(x_1-x_2)^2 \one\\
&=\sum_{i=1}^n x_i^2 H^\ast(\{i\}) + (x_1^2 +x_2^2)\one-
\sum_{i<j}2x_ix_je_{i,j}
-(x_1-x_2)^2\one\\
&=\sum_{i=1}^n x_i^2 H^\ast(\{i\}) -
\sum_{i<j}(2x_ix_j-2x_1x_2)e_{i,j}.
\end{align*}
Next, as in the previous proof, equation \eqref{eq: interval} will give us a representation in terms of the vectors in $\cB^\ast$.
Therefore $\Lat \Im H_X^*\subset\Lat \cB^*$. The argument for the reverse inclusion is the same as in the one-dimensional case. This completes the proof of the first part of the corollary.

For the second part, let us again order the vectors in the convenient block structure as in the proof of Lemma \ref{lem: H^ast}, see
\eqref{eq:matrix M}: the  $(m-1)\times (m-1)$ matrix $M$
defined by
\begin{equation*}
M_{(j,i)}=e_{i,j} \text{ if } i<j<n, \quad
M_{(n,i)}=H([1:i+1]) \text{ for } i<n-1,
\quad
\text{and}
\quad 
M_{(n,n-1)}=\one-H([1])
\end{equation*}
has block representation
\[
M :=
\begin{bNiceArray}{c|c}[margin]
 I_{(m-n)\times(m-n)} & \boldsymbol{0} \\
 \hline
 \ast                & M_1
\end{bNiceArray},
\]
so taking a matrix representation of ${\cal B}_1$ and ${\cal B}_2$, the basis $\cB^*$ can be written in matrix form as
\begin{equation*}
B^*= 
\begin{bNiceArray}{c|c}[margin]
 {\cal B}_1 \otimes I_{(m-n)\times(m-n)} & \boldsymbol{0}\\
 \hline
 \ast                & {\cal B}_2 \otimes M_1
\end{bNiceArray}.
\end{equation*}
Since $M$ is lower diagonal, with 1s along the diagonal, the structure of $B^*$ is block-lower-diagonal with blocks of size $\ell \times \ell$ and by the above representation it has $(m-n)$ blocks ${\cal B}_1$ along the diagonal and $(n-1)$ blocks ${\cal B}_2$ along the diagonal. Consequently
\begin{equation*}
|\Lat \Im H_\mathcal{X}^*|=\det B^* = (\det {\cal B}_1)^{m-n} (\det {\cal B}_2)^{n-1}=
\labsrabs{\Lat \Xc_1}^{m-n}
\cdot\labsrabs{\Lat \Xc_2}^{n-1}.
\end{equation*}  
\end{proof}
To state our final main result, let us introduce 
the following notation for element-wise product:  
for a vector $\boldsymbol{x}=(x_i)$, $\boldsymbol{x}^{\circ 2} :=(x_i^2)$, and similarly for 
a matrix $A = (a_{i,j})$,
$A^{\circ 2} := (a^2_{i,j})$.
\begin{theorem}
    \label{thm: general}
 Let $\boldsymbol{X}_1,  \dots, \boldsymbol{X}_n$ be $d$-dimensional vectors
 whose entries are iid
 samples from a random variable $X$
 with finite  support $\Xc := \supp X$.      
 Let 
$p_d$ be the probability that these vectors are equidistant.
    
 Define the sets
 \begin{equation*}
 \mathcal{X}_1=\{2xy\,|
x, y\in \Xc\}, \text{ and } \mathcal{X}_2= \mathcal{X}_1\cup \{x^2\,|x \in \Xc\},
 \end{equation*}
 and denote by
 $\ell$ the dimension of the vector space over $\Qb$ generated by $\Xc_1$.

Considering
$X$ as a vector
$\boldsymbol{Y}$
in $\Qb^\ell$, pick iid copies 
$\boldsymbol{Y}_{\! 1},
\boldsymbol{Y}_2,
\boldsymbol{Y}_3\sim \boldsymbol{Y}$. Then, with \[
m = \binom{n}{2},
\quad
\text{and}
\quad
\boldsymbol{C}_{\! 1} :=  \mathrm{Cov}((\boldsymbol{Y}_{\! 1} - \boldsymbol{Y}_{\! 2})^{\circ 2}, (\boldsymbol{Y}_{\! 1}-\boldsymbol{Y}_{\! 3})^{\circ 2}) = \Var\lbrb{\lbrb{\boldsymbol{Y}-\Ebb{\boldsymbol{Y}}}^{\circ2}},
\]
we have that as $d \to \infty$,
 \begin{equation}
     \label{eq: general asymp}
 p_d \simi
\frac{1}{d^{\ell(m-1)/2}}\cdot
\frac{
|\Lat \Xc_1|^{m-n}|\Lat \Xc_2|^{n-1}
}
{2^{m-n}
\sqrt{
m^\ell(2\pi)^{\ell(m-1)}\cdot|(\Var\boldsymbol{Y})^{\circ 2}|^{(m-n)}\cdot|4(\Var \boldsymbol{Y})^{\circ 2}+(n-2)\boldsymbol{C}_{\!1}|^{n-1}}
}.
 \end{equation}
\end{theorem}

\begin{remark}
It would be more satisfactory in the last result if we could link in an explicit way $|\Lat \Xc_1|$ and
$|\Lat \Xc_2|$. In the one-dimensional case, consider WLOG that $X$ takes integer values. Then we can take the numbers $2$ and 1 as bases of, respectively, $\Xc_1$ and $\Xc_2$, so $|\Lat \Xc_1| = 2|\Lat \Xc_2| = 2$, which
recovers the result of Theorem \ref{thm: general lattice}, see \eqref{eq: main asymp lattice}. However, in general it turns out that the link between the two fundamental volumes is more subtle. In the special case where every $x^2$ for $x
\in \supp X$ is rational (or, more generally, when the support elements are of the form $s\sqrt{t}$ with $t$ square–free), we show in Lemma~\ref{lemma:rational_squares} that
\[
|\Lat\Xc_2|=2|\Lat\Xc_1|.
\]
Furthermore, in some other special cases we describe a characterisation of the ratio of the two volumes in terms of the Smith normal form of an integer matrix representation of
 the vectors in $\Xc_1$ and $\Xc_2$, see Lemma \ref{lemma:z2_discription}.
\end{remark}

\begin{proof}
In the style of 
\eqref{eq: Krafft_general}, and writing $\boldsymbol{\widetilde{V}}$ for the embedded in
the proper $\ell\times\dim \boldsymbol{V}$-dimensional space vector $\boldsymbol{V}$, we again have
  \begin{equation}
        \label{eq: Krafft_real}
  p_d =   \P \lbrb{\sum_{k=1}^d \boldsymbol{\widetilde{V}}_{\! k} = \mathbf{0}} \simi \frac{|\Lat \boldsymbol{\widetilde{V}}|}{\sqrt{(2\pi d)^{\ell(m-1)}|\Var(\boldsymbol{\widetilde{V}})|}}.
 \end{equation}
First, Corollary \ref{cor:volumes} gives us the desired denumerator, $|\Lat \boldsymbol{\widetilde{V}}| =
|\Lat \Xc_1|^{m-n}|\Lat \Xc_2|^{n-1}$. Second, for the covariance matrix, it is straightforward to check that
the counterpart of
\eqref{eq: Cov Matrix} 
holds with
tensor product.  Then, with the constants, from the theorem statement, that is, with iid
$\boldsymbol{Y}_{\! i} \sim X$, 
\[
\boldsymbol{C}_{\!0}:= \lbrb{\Var\lbrb{
\boldsymbol{Y}_{\! 1} - \boldsymbol{Y}_{\! 2}
}}^{\circ 2},
\quad
\text{and}
\quad
\boldsymbol{C}_{\!1}:= \Cov\lbrb{
\boldsymbol{Y}_{\! 1} -\boldsymbol{Y}_{\! 2}
,
\boldsymbol{Y}_{\! 1}- \boldsymbol{Y}_{\! 3}},
\]
and
\begin{equation}\label{eq: A Matrix_general}\mathcal{A^\ast} = 
\boldsymbol{C}_{\! 0}\otimes\lbrb{ I +
\begin{bNiceArray}{c|w{c}{0.5em}|c}[margin]
  0 & {\one} &{\one}  \\ 
  \hline
  {\mathbf{0}} &
  {\one} &{\one}  \\
  \hline
  {\mathbf{0}} &
  {\one} &{\one} 
\end{bNiceArray}
}
+ \boldsymbol{C}_{\! 1}\otimes\lbrb{\cH
-
\begin{bNiceArray}{c|w{c}{0.5em}|c}[margin]
  0 & {\boldsymbol{0}} &{\mathbf{0}}  \\ 
  \hline
  {\one} &
  {\one} &{\one}  \\
  \hline
  {\mathbf{0}} &
  {\boldsymbol{0}} &{\mathbf{0}} 
\end{bNiceArray}
- 
\begin{bNiceArray}{c|w{c}{0.5em}|c}[margin]
  0 & {\one} &{\mathbf{0}}  \\ 
  \hline
  {\boldsymbol{0}} &
  {\one} &{\boldsymbol{0}}  \\
  \hline
  {\mathbf{0}} &
  {\one} &{\mathbf{0}} 
\end{bNiceArray}
}
=:\boldsymbol{C}_{\! 0}\otimes A_1 +
\boldsymbol{C}_{\! 1}\otimes A_2
,
\end{equation}
we have $
\det {\cA^\ast} = 
\det \boldsymbol{C}_{\! 0} |\Var(\boldsymbol{\widetilde{V}})|$. Generalising the approach from the univariate
case,
in Proposition
\ref{prop: det cA^ast} we calculate $\det \Var(\cA^\ast)$,
which gives
\[
|\Var(\boldsymbol{\widetilde{V}})| = 
m^\ell\labsrabs{\boldsymbol{C}_{\!0}-2\boldsymbol{C}_{\!1}}^{m-n}\labsrabs{\boldsymbol{C}_{\!0} + (n-4)\boldsymbol{C}_{\!1}}^{n-1} .
\]
Moreover,
a direct calculation shows that $
\boldsymbol{C}_{\! 0} - 2 \boldsymbol{C}_{\! 1} = 4 (\Var\boldsymbol{Y})^{\circ 2},
$
so substituting in
\eqref{eq: Krafft_real},
we obtain the asymptotics in the general case
\eqref{eq: general asymp}.
\end{proof}

\subsubsection{Further understanding of the involved lattices}

In this section, we will be interested in the ratio $|\Lat \mathcal{X}_1|/|\Lat \mathcal{X}_2|$. Although, we do not have a good understanding of this ratio for all cases, we present two cases where such relation is relatively easy to obtain. The first one is where the elements of $\Xc$ are square roots of a rational number: in Lemma~\ref{lemma:rational_squares} we show that in this scenario, $|\Lat \Xc_1|=2|\Lat \Xc_2|$.

 The second case captures a more general class of lattices but requires a deeper understanding of the specific representations of the elements from $\mathcal{X}_1$ and $\mathcal{X}_2$. In very general terms, it requires an integer representation of the vectors in $\mathcal{X}_2$. Then by simple scaling with appropriate powers of $2$ we can transform the basis in such a way that we get an integer matrix $A$ representing the vectors in $\mathcal{X}_2$ such that every column of $A$ contains an odd number. Provided that $A$ considered over $\mathbb{Z}_2$ is of \emph{full rank}, in Lemma~\ref{lemma:z2_discription} we show that 
 \[|\Lat \mathcal{X}_1| = 2^{\ell-r}|\Lat \mathcal{X}_2|,\]
 where $\ell$ is the rank of $A$ over $\mathbb{Z}_2$, and $r$ is the rank of the submatrix of $A$ that represents $\{2x_ix_j\,|\, i\neq j\}$ again in $\mathbb{Z}_2$.
  
\begin{lemma}\label{lemma:rational_squares}
In the notation laid at the beginning of Section \ref{subsec: the general case}, assume that for all $i$, $|x_i|=\sqrt{q_i}$ for some non-negative rational numbers $q_i\in \Qb$. Then, $
|\Lat \Xc_1|=2|\Lat \Xc_2|$, and thus in this case the asymptotics from \eqref{eq: general asymp} reads off as 
\[
 p_d \simi
\frac{1}{d^{\ell(m-1)/2}}\cdot
\frac{
|\Lat \Xc_2|^{m-1}
}
{
\sqrt{
m^\ell(2\pi)^{\ell(m-1)}\cdot|(\Var\boldsymbol{Y})^{\circ 2}|^{(m-n)}\cdot|4(\Var \boldsymbol{Y})^{\circ 2}+(n-2)\boldsymbol{C}_{\!1}|^{n-1}}
}.
\]
\end{lemma}
\begin{proof}
Let $q_i=a_i/b_i$ with $a_i, b_i\in \mathbb{Z}$ and let $b=b_1 b_2 \dots b_k$. Then multiplying all $x_i$ by $b$ preserves the ratio of the desired volumes, whereas $bq_i=b'_i \sqrt{a_ib_i}$ with $b'_i=b/b_i$ is a positive integer.
Therefore, it suffices to consider the case where all $x_i$ are in the special form $x_i=s_i\sqrt{t_i}$  with $s_i,t_i\in \mathbb{N}$, and all $t_i$ are square free. Note that in the case where $x_ix_j\in \mathbb{N}$, $\gcd(x_i^2,x_ix_j,x_j^2)=gcd(x_i^2,x_j^2)$ and consequently $\gcd(2x_i^2,2x_ix_j,2x_j^2)=2\gcd(x_i^2,x_j^2)$. Therefore,
\begin{equation}\label{eq: 2 times}
\gcd\lbrb{\{2x_i^2\,|\, i\le k\} \cup \{2x_ix_j\,|\, x_ix_j\in \mathbb{N}\}}=2\gcd(\{x_i^2\,|\, i\le k\}),
\end{equation} 
so in order to conclude the proof, it suffices to notice that if $t_i> 1$ are distinct square-free positive integers, then if for some integers $s'_i$
\begin{equation}\label{rational_independence}
\sum_{i=1}^{k} s_i' \sqrt{t_i}\in \mathbb{Q},
\quad
\text{ then all $
s_i'=0$}.
\end{equation}
 The last Galois-type result seems to be well-known in the field of algebra, however for completeness, a short proof is available in Proposition \ref{prop: rational_independence} of the Appendix. 
Therefore, the set $\{2x_ix_j\,|\, x_ix_j\not \in \mathbb{N}\}$ spans a linear subspace of ${\cal L}$ that is orthogonal to $\boldsymbol{1}$ -- the vector that is collinear with all rational numbers.
Hence, the ratio of the fundamental volumes of the lattices $\Lat \mathcal{X}_1$ and $\Lat \mathcal{X}_2$ is equal to the ratio of their respective projections on $\mathbb{Q}$ (as a linear subspace of ${\cal L}$), and thus as we have seen above in \eqref{eq: 2 times} is equal to $2$, as claimed in the lemma.
\end{proof}

We continue with the second scenario that we present. Let  $v_i$ and $u_{i,j}$ be the coordinates in
$\cB_{ref}$ of, respectively, $x_i^2$ and $2x_ix_j$ for $x_i,x_j \in \Xc$. 
Then, leaving out the representations of $0=x_0^2 = 2x_0x_i$,  we define the rational matrix $A$ of size $k^2\times \ell$ by
\begin{equation}
    \label{eq: def UV}
A \;:=\;
\begin{bNiceMatrix}
  \text{--- }u_1\text{ ---} \\
  \vdots \\
  \text{--- }u_k\text{ ---} \\
  \text{--- }v_{1,2}\text{ ---} \\
  \vdots \\
  \text{--- }v_{k-1,k}\text{ ---}
\end{bNiceMatrix}
\;=:\;
\begin{bNiceMatrix}[hvlines,margin]
  U \in M_{k \times \ell}(\Zb) \\
  {V\in M_{(k^2-k) \times \ell}(\Zb)}
\end{bNiceMatrix},
\end{equation}
that is, with rows $A_i=v_i$ for $i\le k$ and $A_{ik+j}=u_{i,j}$ for $i\neq j$. 

Scaling the elements of $A$ with an appropriate constant, say the $\lcm$ of the denominators of the entries of $A$, we may assume that $A$ is actually integer-valued. Note that this transformation corresponds to scaling the standard basis of $\mathbb{Q}^{\ell}$ by a constant, which is an invariant operation for the purposes of our problem.

Next, let $\rho_{s}$ be the maximal power of $2$ that divides all the entries in the $s$th for $s\leq \ell$ column of $A$, that is, for $s\le \ell$,
\begin{equation*}
\rho_{s}:=\max \{\rho\in \mathbb{N} :  2^{\rho} \mid A_{i,s} \text{ for all } i\le k^2\}.
\end{equation*}
Then further scaling the $s$th column of $A$ by $2^{-\rho_s}$, we effectively scale the $s$th dimension of the basis, so that the event in which we are interested remains the same.

The described transformations show that we can assume that in each column of $A$ there is at least one odd number. Now we can state the second case for which we can say something about the ratio $|\Lat \Xc_1|/|\Lat \Xc_2|$ in terms of ranks of ${U}$ and ${V}$ from \eqref{eq: def UV}.
For this reason, for any matrix $M$ with integer entries, we define $M \pmodsmall{2}$ to be the matrix obtained by reducing each entry of $M$ modulo 2.
\begin{lemma}\label{lemma:z2_discription}
Fix the notation laid at the beginning of Section \ref{subsec: the general case} and $A,U,V$ as above.
Assume that $A \pmodsmall{2}$ has full rank, $\ell$, over $\mathbb{Z}_2$; and that
$V \pmodsmall{2}$ has rank $r$. Then $|\Lat \Xc_1|=2^{\ell-r} |\Lat \Xc_2|$ and further
\begin{equation*}
|\Lat \Im H^*_X| = 
 |\Lat \Xc_1|^{m-n}\cdot 
|\Lat \Xc_2|^{n-1}
= 2^{(m-n)(r-\ell)} |\Lat \Xc_2|^{m-1}.
\end{equation*}
\end{lemma}
\begin{proof}
Recall that $V$, and respectively $V\pmodsmall{2}$, corresponds to representations of $2x_ix_j$  with $i\neq j$, and define the
matrix
\[
\widehat{A} \;:=\;
\begin{bNiceMatrix}[hvlines,margin]
  {2U} \\
  {V}
\end{bNiceMatrix},
\quad
\text{so}
\quad
\widehat{A} \pmodsmall{2}\;=\;
\begin{bNiceMatrix}[hvlines,margin]
  \boldsymbol{0} \\
  {V\pmodsmall{2}}
\end{bNiceMatrix}.
\]
Therefore $\widehat{A}$ represents the set $\mathcal{X}_1:= \{2x_ix_j: 1\leq i,j \leq n\}$ in the same basis as $A$ represents $\mathcal{X}_2:=\{
x_i^2:1\leq i \leq n
\} \cup 
\Xc_1$. Furthermore, by the conditions of the lemma, $\rank({A} \pmodsmall{2})=\ell$ and $\rank(\widehat{A} \pmodsmall{2})=r$.

Let us introduce the following notation: for a set of rows $I$ and a set of columns $J$ of a matrix $M$, let $M_{I,J}$  be the submatrix of $M$ determined by the rows $I$ and columns $J$. Then, it is straightforward that
\begin{equation*}
\gcd\lbrb{\widehat{A}_{I,J}} \in \lbcurlyrbcurly{2\gcd(A_{I,J}),\gcd(A_{I,J})}.
\end{equation*}
Furthermore if the matrix $\widehat{A}_{I,J}$ contains an odd element, then $\gcd(\widehat{A}_{I,J})=\gcd(A_{I,J})$. 

To prove that the fundamental volume defined by the lattice $\widehat{A}$ is $2^{\ell-r}$ times the volume defined by $A$ we consider how the following algorithm transforms an integer matrix $A\in M_{k^2,\ell}(\mathbb{Z})$ in Smith's normal form:
\begin{enumerate}
\item Using elementary matrix operations (that is add a row/column to a row/column and switch of rows/columns) transform $A=A(0)$ to a matrix:
\begin{equation*}
    \widetilde{A}:=
\begin{bNiceArray}{c|c}[hvlines, margin]
  \gcd(A) & \boldsymbol{0}\\
  \boldsymbol{0}       & A(1)
\end{bNiceArray}
.
\end{equation*}
\item Apply step 1 to $A{(1)}$ recursively.
\item the fundamental volume of $A$ is then the product of the non-zero elements of $A{(\ell-1)}$, that is, $\prod_{j=0}^{\ell-1}\gcd(A(j))$.
\end{enumerate}
Note that applying the same elementary operations to a matrix $A\pmodsmall{2}$ preserves its rank and maintains the invariant $(A\pmodsmall{2})(j)=A(j)\pmodsmall{2}$. 

     Now it is straightforward that $\gcd(\widehat{A}(j))=\gcd(A(j))$ for $j< r$ because the rank of $\widehat{A}\pmodsmall{2}$ is $r$, and thus the greatest common divisors in the first $r$ rounds will be odd. On the other hand, in the subsequent $\ell-r$ rounds $j$, $\gcd(\widehat{A}(j))=2\gcd({A}(j))$, so the fundamental volume of $\widehat{A}$ is $2^{\ell-r}$ times the fundamental volume of $A$.

\end{proof}
\appendix
\section{Appendix}\label{sec: appn}

\subsection*{The line graph of the complete graph}
For a graph $G(V,E)$, its line graph 
$L(G)$ is the graph with vertex set $E$, and such that two of its
vertices are connected if
the respective edges in $G$ have 
a point in common.

Let $K_n$ be the complete graph on
$n$ vertices. Its line graph
$L(K_n)$ is also known as the \emph{triangular} graph, or also
the Johnson graph $J(n,2)$. The eigenvalues of line graphs $L(G)$ are
easily obtained from the eigenvalues of $G$, see e.g.
\cite[Proposition 1.4.1]{Brouwer-2012}. However, we have not found an
explicit description of the respective eigenvectors, which we shortly present in the following proposition.
\begin{proposition} \label{prop: line Kn} Denote with $m:=\binom{n}{2}$.
 The eigenvalues of the triangular graph $L(K_n)$ are $(2n-4)^1, (n-4)^{n-1}$, and $(-2)^{m-n}$, where the superscripts denote multiplicities. Moreover,
 the respective eigenspaces have bases
 \begin{itemize}
     \item the vector of 1s, which we denote with $\one$;
     \item $(n-1)$ distinct vectors from the type $H_i-\frac{2}{n}\one$, with $H_i:=H(\{i\})$ defined via \eqref{def: overlap};
     \item $(m -n)$ vectors which are in bijection with a basis of the even-length cycles of $K_n$ through the map
     \[
\text{a cycle } (i_1, \dots, i_{2k},i_1)\in V^{2k+1} \mapsto
    e_{i_1i_2}- e_{i_2i_3}+\dots+ e_{i_{2n-1}i_{2n}} - e_{i_{2n}i_1} \in \R^m
     \]
     Moreover, it is possible to construct a basis of this space which contains $e_{12}-e_{23}+e_{34}-e_{14}$ and $(n-1)$ other vectors, orthogonal to $e_{1,2}$.
 \end{itemize}
\end{proposition}
\begin{remark}
    \label{rem: Doob} As we will shortly see from the proof of the last proposition, the
    eigenvalue $(-2)$ is characteristic
    for all line graphs. In particular, for a connected graph with  $n$ vertices and $m$ edges, $(-2)$ is an eigenvalue of multiplicity $m-r$. The work of \cite{Doob-1973} characterises the associated with
    $(-2)$ eigenvectors in terms
    of the dendroid of the matroid of $G$, which in the specific case of $K_n$ gives the characterisation in terms of the even cycles of $K_n$.

    If one wants to build an explicit basis the space of even cycles, it is possible to
    adapt
    the
    concept of \emph{Fundamental Cycles}, 
   see e.g. \cite{Gross-Yellen-Anderson-2019}[Section 4.5]: consider a connected graph $G$. Its cycle space $C(G)$ is
    the subset of the edge space $E$ consisting all cycles of $G$ and all unions of edge-disjoint cycles. 
If $T$ is a spanning tree of $G$, one may construct
the \emph{Fundamental system of cycles, associated of T} as the
$|E|-|V|+1$ number of cycles which are obtained by connecting each edge which is not in $T$ with the spanning tree itself.

In the context of Proposition \ref{prop: line Kn}, we are interested only in cycles of even length. We can proceed as in \cite{Doob-1973}[Proposition 3.1] by first selecting only the even cycles from a fundamental cycle basis. Then, fix an odd reference cycle $C_B$. For each other odd cycle $C$, we can construct a basis element by first traversing $C_B$, then following a path in the spanning tree $T$ to $C$, traversing $C$, and finally returning to $C_B$.

    \end{remark}
        \begin{proof}
        [Proof of Proposition \ref{prop: line Kn}]
Consider the incidence matrix $M $ of $ K_n $. This is an $n \times m$ matrix, such that $ M_{v,e} = 1 $ if vertex $v$ is incident to edge $ e $, and $ 0 $ otherwise. Denoting by $A(K_n)$
the adjacency matrix of $K_n$, with $J_n$ the matrix of 1s, and with $I_n$ the identity, one can check that
$MM^t = A(K_n) = J_n + (n-2)I_n$, which has eigenvalues $(2n-2)^{1}$ and $(n-2)^{n-1}$.
However, $M^tM = \cH+ 2I_m$, and it is a standard fact form linear algebra that the eigenvalues of 
$MM^t$ and $M^tM$ coincide up to the the addition of $0^{m-n}$. Therefore the eigenvalues of $H$
are $(2n-4)^1, (n-4)^{n-1},$ and $(-2)^{m-n}$ as
claimed in the proposition.

Let us construct respective eigenvectors.
For $(2n-2)^{1}$, it is clear that $H \one = (2n-4)\one$. 

For $(n-2)^{n-1}$, we have from \eqref{eq: Hi+Hj - Hij} and a direct calculation (the representation from Figure \ref{fig: ij} may be helpful) that
\[
H_i+H_j-\cH_{i,j} = 2e_{ij},
\quad
\text{and}
\quad
\langle H_i, H_j\rangle = 
1 + (n-2)\ind{i=j}
=\begin{cases}
    n-1 & \text{if } i=j\\
    1 & \text{otherwise}
\end{cases}.
\]
Therefore $(\cH H_i)_{k,\ell} = \langle H_{k} + H_\ell-2e_{k,\ell} , H_{i}\rangle = 2 + (n-2-2)\ind{\text{$k$ or $\ell$ is equal to $i$}}$, which in 
vector form means that 
\begin{equation}
    \label{eq: HHi}
\cH H_i = 2\cdot\one + (n-4)H_i,
\end{equation}
so 
\[
\cH \lbrb{H_i - \frac{2}{n}\one} = 2\cdot\one + (n-4)H_i - \frac{2}{n}(2n-4)\one = (n-4)\lbrb{H_i - \frac{2}{n}\one},
\]
giving us indeed that $H_i - \frac{2}{n}\one$ is an eigenvector
with eigenvalue $(n-4)$. We know from the first part of the proposition that $(n-4)$ has multiplicity $(n-1)$, so it remains to show that any $(n-1)$ vectors of type $H_i - \frac{2}{n}\one$ are linearly independent. Indeed, let us argue for $H_1, \dots, H_{n-1}$. Assume that for some coefficients $\alpha_i$, 
\[
\sum_{k=1}^{n-1} \alpha_k \lbrb{ H_k - \frac{2}{n}\one } = \mathbf{0},
\quad
\text{which implies that}
\quad
\sum_{k=1}^{n-1} \alpha_k H_k = \left( \frac{2}{n} \sum_{k=1}^{n-1} \alpha_k \right) \one.
\]
However, only $H_i$ has a non-zero coordinate along $e_{i,n}$, so $\alpha_i = \frac{2}{n} \sum_{k=1}^{n-1} \alpha_k$, and this easily leads to all $\alpha_i$ being zero, ensuring linear independence.

For the last part, for an even length cycle of $K_n$, $(i_1, \dots, i_{2k},i_1)$, we have that
\begin{align*}
\cH( e_{i_1,i_2}- e_{i_2,i_3}&+\dots - e_{i_{2k},i_1})
=
\cH_{i_1,i_2} - \cH_{i_2,i_3} + \dots -
\cH_{i_{2k,1}} \\
&\quad= H_1 + H_2 - 2e_{1,2} - (H_2 + H_3 - 2e_{2,3}) + \dots - (H_{2k} + H_1-2e_{2k,1})\\
&\quad=- 2\lbrb{e_{i_1,i_2}- e_{i_2,i_3}+\dots+ e_{i_{2k-1},i_{2k}} - e_{i_{2k},i_1}}.
    \end{align*}
It remains to confirm that the dimension of
linear span of these type of vectors is $m-n$,
which can be seen by the construction of basis
sketched in Remark \ref{rem: Doob}.
    \end{proof}

     \section*{Auxiliary results}
     \begin{proposition}
         \label{prop: det cA}
The determinant of the matrix $\cA$, defined in \eqref{eq: detSigma} is equal to
\[
\det \cA = mC_0\lbrb{C_0-2C_1}^{m-n}\lbrb{C_0 + C_1(n-4)}^{n-1} .
\]
     \end{proposition}
\begin{proof}
Recall that equation \eqref{eq: A Matrix} provided that
\[\mathcal{A} = 
C_0\lbrb{ I +
\begin{bNiceArray}{c|w{c}{0.5em}|c}[margin]
  0 & {\mathbf{1}} &{\mathbf{1}}  \\ 
  \hline
  {\mathbf{0}} &
  {\one} &{\one}  \\
  \hline
  {\mathbf{0}} &
  {\one} &{\one} 
\end{bNiceArray}
}
+ C_1\lbrb{\cH
-
\begin{bNiceArray}{c|w{c}{0.5em}|c}[margin]
  0 & {\boldsymbol{0}} &{\mathbf{0}}  \\ 
  \hline
  {\one} &
  {\one} &{\one}  \\
  \hline
  {\mathbf{0}} &
  {\boldsymbol{0}} &{\mathbf{0}} 
\end{bNiceArray}
- 
\begin{bNiceArray}{c|w{c}{0.5em}|c}[margin]
  0 & {\one} &{\mathbf{0}}  \\ 
  \hline
  {\boldsymbol{0}} &
  {\one} &{\boldsymbol{0}}  \\
  \hline
  {\mathbf{0}} &
  {\one} &{\mathbf{0}} 
\end{bNiceArray}
}
.
\]
where the matrix and its blocks are symmetric and
\begin{itemize}
    \item The first line describes row $(1,2)$;
    \item The second line describes rows $(1,2),(1,3), \dots, (1,n),(2,3), (2,4), \dots, (2,n)$, that
    is the positions where $\cH_{1,2}$ is non-zero;
    \item The last line is for the remaining indices.
\end{itemize}
We have seen in the proof of Theorem \ref{thm: general} that
the set
\[
\mathcal{B} := 
\lbcurlyrbcurly{\one, \cH_{1,2}, H_2, H_3, \dots, H_n},
\]
consists of $(n+1)$ linearly independent vectors, and its orthogonal complement consists of eigenvectors of $\cH$ with eigenvalue
$C_0 - 2C_1$, and has dimension $m-(n+1)$. Therefore,
\begin{equation}
    \label{eq: detA AB}
    \det \cA = (C_0-2C_1)^{m-(n+1)} \det \cA_{|\cB},
\end{equation}
where
$A_{|\cB}$ is the matrix of the restriction of $\cA$
to the linear span of $\mathcal{B}$ in the basis $\cB$. 
Recall that  from previous calculations that
\[
\cH_{i,j} \eqinfo{\eqref{eq: Hi+Hj - Hij}}
H_i+H_j-2e_{i,j}, 
\quad
\cH H_i \eqinfo{\eqref{eq: HHi}}
2 \one + (n-4)H_i,
\quad
\text{and}
\quad
\sum_{i=1}^nH_i \eqinfo{\eqref{eq: 2 one}} 2\cdot\one.
\] 
This helps us to calculate that
\[
\cH \cH_{1,2} = \cH (H_1 + H_2 - 2e_{1,2}) = (n-6)\cH_{1,2} + 4 \boldsymbol{1} + (2n-8)e_{1,2},
\]
and further that
\begin{equation}
\label{eq: calc A}
\begin{split}
    \mathcal{A}\boldsymbol{1} &= C_0 m \boldsymbol{1} - C_1 m \cH_{1,2}
    \\
     \mathcal{A}{\cH_{1,2}}
     &= 
     \big( (2n - 4) C_0 - (2n - 8) C_1 \big) \one + \big( C_0 - (n + 2) C_1 \big) \cH_{1,2} + C_1(2n-8)e_{1,2}\\
     &= \big( (2n - 4) C_0 - (2n - 8) C_1 \big) \one + \big( C_0 - (n + 2) C_1 \big) \cH_{1,2} \\
     &\hspace{15em}+ C_1(2n-8)\frac12\lbrb{2 \boldsymbol{1} - \cH_{1,2}-\sum_{i>2}H_i}\\
     &= 
     (2n-4)C_0\one + (C_0 -(2n-2)C_1)\cH_{1,2}
     -C_1(n-4)\sum_{i>2}H_i\\
     \cA H_2&= (C_0 (n-2) - C_1(n-4))\one-C_1(n-1)\cH_{1,2} + (C_0+(n-4)C_1)H_2\\
     \text{and, for $i>2$, }\quad
     \cA H_i&= C_0 (n-1)\one-C_1(n-1)\cH_{1,2} + (C_0+(n-4)C_1)H_i.
\end{split}
\end{equation}
Therefore, the matrix of $\cA_{|\cB}$ is 
\[
\scalebox{0.8}{
  $\begin{bNiceMatrix}[first-col, margin]
  \one    & C_0m                & C_0(2n-4)            & C_0(n-2) - C_1(n-4)   & C_0(n-1)         & \dots & C_0(n-1) \\
  \cH_{1,2} & -C_1m              & C_0 - (2n-2)C_1     & -C_1(n-1)             & -C_1(n-1)        & \dots & -C_1(n-1)\\
  H_2    & 0                  & 0                    & C_0 + C_1(n-4)        & 0                & \dots & 0 \\
  H_3    & 0                  & -C_1(n-4)            & 0                     & C_0 + C_1(n-4)   & \dots & 0 \\
  \vdots & \vdots             & \vdots               & \vdots                & \vdots           & \ddots & \vdots\\
  H_n    & 0                  & -C_1(n-4)            & 0                     & 0                & \dots & C_0 + C_1(n-4)
  \end{bNiceMatrix}$
}
\]
To calculate its determinant, first use the cofactor formula with respect to the third line, to get
\[
\det \cA_{|\cB} = 
 \lbrb{C_0 + C_1(n-4)} \scalebox{0.8}{ $\begin{vNiceMatrix}

   C_0m                & C_0(2n-4)             & C_0(n-1)         & \dots & C_0(n-1) \\
   -C_1m              & C_0 - (2n-2)C_1               & -C_1(n-1)        & \dots & -C_1(n-1)\\
  0                  & -C_1(n-4)                           & C_0 + C_1(n-4)   & \dots & 0 \\
   \vdots             & \vdots                         & \vdots           & \ddots & \vdots\\
   0                  & -C_1(n-4)                         & 0                & \dots & C_0 + C_1(n-4)
  \end{vNiceMatrix}$},
\]
and then since $m = \binom{n}{2}$, multiplying the first column by
$-2/n$ and adding it to columns $3,4,\dots$, we obtain
\begin{equation}\label{eq: calc det Ab}
\begin{split}
\det(\cA_{|\cB}) 
&= \lbrb{C_0 + C_1(n-4)} 
   \scalebox{0.8}{$
     \begin{vNiceMatrix}
       C_0m    & C_0(2n-4)       & 0                  & \dots & 0 \\
       -C_1m   & C_0 - (2n-2)C_1 & 0                  & \dots & 0 \\
       0       & -C_1(n-4)       & C_0 + C_1(n-4)     & \dots & 0 \\
       \vdots  & \vdots          & \vdots            & \ddots & \vdots \\
       0       & -C_1(n-4)       & 0                  & \dots & C_0 + C_1(n-4)
     \end{vNiceMatrix}$
   } \\
&= \lbrb{C_0 + C_1(n-4)}^{n-1}
   \scalebox{0.8}{$
     \begin{vNiceMatrix}
       C_0m    & C_0(2n-4) \\
       -C_1m   & C_0 - (2n-2)C_1 
     \end{vNiceMatrix}$
   }\\
   &=mC_0(C_0-2C_1)\lbrb{C_0 + C_1(n-4)}^{n-1} .
\end{split}
\end{equation}
Substituting in \eqref{eq: detA AB}, we obtain the claimed
\[
\det \cA = mC_0\lbrb{C_0-2C_1}^{m-n}\lbrb{C_0 + C_1(n-4)}^{n-1} .
\]
\end{proof}
     \begin{proposition}
         \label{prop: det cA^ast}
The determinant of the matrix $\cA^\ast$, defined in \eqref{eq: A Matrix_general} is equal to
\[
\det \cA^\ast = m^\ell|\boldsymbol{C}_{\!0}|\labsrabs{\boldsymbol{C}_{\!0}-2\boldsymbol{C}_{\!1}}^{m-n}\labsrabs{\boldsymbol{C}_{\!0} + (n-4)\boldsymbol{C}_{\!1}}^{n-1} .
\]
     \end{proposition}
     \begin{proof}
Recall that from \eqref{eq: A Matrix_general}
\[\mathcal{A^\ast} = 
\boldsymbol{C}_{\!0}\otimes\lbrb{ I +
\begin{bNiceArray}{c|w{c}{0.5em}|c}[margin]
  0 & {\one} &{\one}  \\ 
  \hline
  {\mathbf{0}} &
  {\one} &{\one}  \\
  \hline
  {\mathbf{0}} &
  {\one} &{\one} 
\end{bNiceArray}
}
+ \boldsymbol{C}_{\!1}\otimes\lbrb{\cH
-
\begin{bNiceArray}{c|w{c}{0.5em}|c}[margin]
  0 & {\boldsymbol{0}} &{\mathbf{0}}  \\ 
  \hline
  {\one} &
  {\one} &{\one}  \\
  \hline
  {\mathbf{0}} &
  {\boldsymbol{0}} &{\mathbf{0}} 
\end{bNiceArray}
- 
\begin{bNiceArray}{c|w{c}{0.5em}|c}[margin]
  0 & {\one} &{\mathbf{0}}  \\ 
  \hline
  {\boldsymbol{0}} &
  {\one} &{\boldsymbol{0}}  \\
  \hline
  {\mathbf{0}} &
  {\one} &{\mathbf{0}} 
\end{bNiceArray}
}
=:\boldsymbol{C}_{\!0}\otimes A_1 +
\boldsymbol{C}_{\!1}\otimes A_2
.
\]
     \end{proof}
     Next, take a basis of reference $\cB_R$ of $\R^\ell$. Then, as in the previous proposition
the set
\[
\mathcal{B} := 
\lbcurlyrbcurly{\one, \cH_{1,2}, H_2, H_3, \dots, H_n},
\]
consists of $(n+1)$ linearly independent vectors, and its orthogonal complement consists of common eigenvectors of $A_1$ and $A_2$ with eigenvalues, respectively,
$C_0$ and $ 2C_1$, and has dimension $m-(n+1)$. Therefore,
\[
    \det \cA^\ast = |\boldsymbol{C}_{\!0}-2\boldsymbol{C}_{\!1}|^{m-(n+1)} \cdot \det \cA^\ast_{|\cB_R\otimes\cB},
\]
where
$\cA^\ast_{|\cB_R\otimes\cB}$ is the matrix of the restriction of $\cA^\ast$
to the linear span of $\Bc_R\otimes\mathcal{B}$ in the basis $\Bc_R\otimes\mathcal{B}$. Furthermore, the analogue of \eqref{eq: calc A} is the following: for a vector $v \in \Bc_R$,
\[ \cA^\ast \lbrb{v\otimes \one}
=\boldsymbol{C}_{\! 0}v\otimes(m\one) + \boldsymbol{C}_{\! 1}v\otimes\lbrb{-2m\Hc_{1,2}}, \quad
\text{etc...}
\]
showing that, the structure from the univariate case is preserved, with elements becoming blocks
\[
\scalebox{0.8}{
  $\begin{bNiceMatrix}[first-col, margin]
  \cB_R\otimes\one    & \boldsymbol{C}_{\! 0}m                & \boldsymbol{C}_{\! 0}(2n-4)            & \boldsymbol{C}_{\! 0}(n-2) - \boldsymbol{C}_{\! 1}(n-4)   & \boldsymbol{C}_{\! 0}(n-1)         & \dots & \boldsymbol{C}_{\! 0}(n-1) \\
  \cB_R\otimes\cH_{1,2} & -\boldsymbol{C}_{\! 1}m              & \boldsymbol{C}_{\! 0} - (2n-2)\boldsymbol{C}_{\! 1}     & -\boldsymbol{C}_{\! 1}(n-1)             & -\boldsymbol{C}_{\! 1}(n-1)        & \dots & -\boldsymbol{C}_{\! 1}(n-1)\\
 \cB_R\otimes H_2    & 0                  & 0                    & \boldsymbol{C}_{\! 0} + \boldsymbol{C}_{\! 1}(n-4)        & 0                & \dots & 0 \\
\cB_R\otimes  H_3    & 0                  & -\boldsymbol{C}_{\! 1}(n-4)            & 0                     & \boldsymbol{C}_{\! 0} + \boldsymbol{C}_{\! 1}(n-4)   & \dots & 0 \\
  \vdots & \vdots             & \vdots               & \vdots                & \vdots           & \ddots & \vdots\\
\cB_R\otimes  H_n    & 0                  & -\boldsymbol{C}_{\! 1}(n-4)            & 0                     & 0                & \dots & \boldsymbol{C}_{\! 0} + \boldsymbol{C}_{\! 1}(n-4)
  \end{bNiceMatrix}$
}.
\]
The logic of \eqref{eq: calc det Ab}, so
\begin{equation*}
\begin{split}
\det(\cA^\ast_{|\cB}) 
&= \labsrabs{\boldsymbol{C}_{\! 0} + \boldsymbol{C}_{\! 1}(n-4)}^{n-1}
   \scalebox{0.8}{$
     \begin{vNiceMatrix}
       \boldsymbol{C}_{\! 0}m    & \boldsymbol{C}_{\! 0}(2n-4) \\
       -\boldsymbol{C}_{\! 1}m   & \boldsymbol{C}_{\! 0} - (2n-2)\boldsymbol{C}_{\! 1} 
     \end{vNiceMatrix}$
   }\\
   &=\labsrabs{\boldsymbol{C}_{\! 0} + \boldsymbol{C}_{\! 1}(n-4)}^{n-1}m^\ell
    \scalebox{0.8}{$
     \begin{vNiceMatrix}
       \boldsymbol{C}_{\! 0}    & \boldsymbol{C}_{\! 0}(2n-4) \\
       -\boldsymbol{C}_{\! 1}   & \boldsymbol{C}_{\! 0} - (2n-2)\boldsymbol{C}_{\! 1} 
     \end{vNiceMatrix}$ }
     \\
   &=\labsrabs{\boldsymbol{C}_{\! 0} + \boldsymbol{C}_{\! 1}(n-4)}^{n-1}m^\ell
    \scalebox{0.8}{$
     \begin{vNiceMatrix}
       \boldsymbol{C}_{\! 0}    & \boldsymbol{0} \\
       -\boldsymbol{C}_{\! 1}   & \boldsymbol{C}_{\! 0} - 2\boldsymbol{C}_{\! 1} 
     \end{vNiceMatrix}$ } =
     \labsrabs{\boldsymbol{C}_{\! 0} + \boldsymbol{C}_{\! 1}(n-4)}^{n-1}m^\ell
    \labsrabs{
       \boldsymbol{C}_{\! 0}}\cdot\labsrabs{   \boldsymbol{C}_{\! 0} - 2\boldsymbol{C}_{\! 1}} ,
\end{split}
\end{equation*}
as claimed.

    \begin{proposition}
    \label{prop: pachi krak}
Let $A$ be a $n \times n$ matrix with diagonal terms equal to $a$ and off-diagonal terms equal to b. Then
$\det A = (a + (n-1)b)(a-b)^{n-1}$. 

\end{proposition}

\begin{proof}
    This can be obtained by standard linear algebra
    for manipulating determinants.

    However, it is also possible to 
    argue as follows:
    if  $J_n$ is the matrix only with 1s and $I_n$ is the identity one,
    then
    $A = (a-b)I_n + bJ_n$.
    Since $J_n$ has eigenvalues
    $1$ and $0$ with
    multiplicities,
    respectively,
    $n-1$ and 1, we get that the eigenvalues of
    $A$ are $a + (n-1)b$ and 
    $a-b$ of the same multiplicities, which gives the desired result.
    
\end{proof}
The ideas of the next result are discussed, for example, in 
\href{https://math.stackexchange.com/questions/4178709/adding-finite-list-of-square-roots-of-primes-to-mathbbq}{this Math Stack Exchange post}\footnote{https://math.stackexchange.com/questions/4178709/adding-finite-list-of-square-roots-of-primes-to-mathbbq}.

     \begin{proposition}
         \label{prop: rational_independence}
         \begin{enumerate}
             \item For any distinct prime numbers $p_1,p_2,\dots,p_n$, if $p$ is square-free and co-prime with $p_1\dots p_n$, then $\sqrt{p}\not \in K_n=\mathbb{Q}[\sqrt{p}_1,\dots,\sqrt{p}_n]$. In particular, $|K_n:\Qb| = 2^n$
             and the numbers $\prod_{i\in I\subset\{1,\dots,n\}}\sqrt{p_i}$ are linearly independent over $\mathbb{Q}$.
             \item           If $t_i> 1$ are distinct square-free positive integers, then the only integers $s'_i$ with the property
$
\sum_{i=1}^{\ell} s_i' \sqrt{t_i}\in \mathbb{Q} \text{ are } s_i'=0 \text{ for all } i\le \ell.
$
         \end{enumerate}
\end{proposition}

\begin{proof}
    The proof of the first claim proceeds by induction on $n$: assume that $p_n$ is prime and is distinct from $p_1,\dots,p_{n-1}$. Then, in particular, $p_n$ is co-prime with $p_1,\dots,p_{n-1}$, and consequently by the inductive hypothesis $p_n\not \in K_{n-1}$, which implies that $|K_{n-1}[\sqrt{p_n}]:K_{n-1}|=2$. Consider the only non-trivial automorphism $\sigma\in Gal(K_n/K_{n-1})$. Then $\sigma(\sqrt{p_n})=-\sqrt{p_n}$. Now consider an arbitrary number $p$ that is square-free and coprime with $p_1,\dots,p_{n-1},p_n$. Thus $\sqrt{p}\not \in K_{n-1}$. For the sake of contradiction, assume that $\sqrt{p}\in K_n$. Thus $\sqrt{p}=a+b\sqrt{p_n}$ with $a\in K_{n-1}$ and $b\neq 0$. It follows that $\sigma(\sqrt{p})=a-b\sqrt{p_n}$ and consequently $(a+b\sqrt{p_n})^2=p=\sigma(\sqrt{p}^2)=\sigma(\sqrt{p})^2=a^2 -2ab\sqrt{p_n} +p_n$. It follows that $ab=0$. Since $b\neq 0$, $a=0$ and thus $\sqrt{pp_n}\in K_{n-1}$. Since $\gcd(p,p_n)=1$, $pp_n$ is square-free and further $pp_n$ is co-prime with $p_1,\dots,p_{n-1}$, this is a contradiction with the inductive hypothesis, applied to $pp_{n}$.

    Therefore, $|K_n : \mathbb{Q}|=2^n$, that is the dimension of $K_n$ as a linear space over $\mathbb{Q}$. On the other hand every element of $K_n$ is a rational combination of $\prod_{i\in I}\sqrt{p_i}$ for $I\subseteq [n]$. This shows that the terms $\prod_{i\in I}\sqrt{p_i}$, being $2^n$ in total, must be linearly independent over $\mathbb{Q}$.
In our case, taking $p_1,\dots,p_n$ to be the distinct prime factors of $t_i$, we see that $t_i$ has the above form and thus are linearly independent over $\mathbb{Q}$.
\end{proof}

\section*{Acknowledgments}
This study is financed by the European Union's NextGenerationEU, through the National Recovery and Resilience Plan of the Republic of Bulgaria, project No BG-RRP-2.004-0008. The second author thanks Ivailo Hartarsky for his useful remarks and overall interest.
\addcontentsline{toc}{section}{References}

\bibliographystyle{plainnat}

\bibliography{bibliography.bib}

\end{document}